\documentclass[a4paper,11pt]{amsart}

\usepackage[english]{babel} %

\usepackage{lmodern}
\usepackage[T1]{fontenc}

\usepackage{amssymb}
\usepackage{mathtools} %
\usepackage{slashed} %
\usepackage[babel=true]{microtype} %
\usepackage[autostyle=true]{csquotes} %
\usepackage[dvipsnames]{xcolor} %

\calclayout

\definecolor{Heather}{RGB}{164, 132, 172}
\usepackage[pdfusetitle,bookmarks,pdfpagelabels]{hyperref} %
\hypersetup{
  colorlinks,
  urlcolor=Periwinkle,
  citecolor=Heather,
  linkcolor=teal,
  breaklinks=true,
  final,
}

\usepackage[
    citestyle=numeric-comp,
    bibstyle=numeric-comp,
    backend=biber,
    url=true,
    doi=true,
    isbn=false,
    giveninits=true,
    maxbibnames=100,
    maxcitenames=10,
    sorting=nyt
]{biblatex} 

\DeclareDelimFormat[textcite]{multinamedelim}{\textendash}
\DeclareDelimAlias[textcite]{finalnamedelim}{multinamedelim}

\DeclareSourcemap{
  \maps{
    \map{ %
        \step[fieldset = month, null]
        \step[fieldset = day, null]
        \step[fieldset = language, null]
    }
    \map{ %
        \step[fieldsource = doi, final]
        \step[fieldset = url, null]
    }
    \map{ %
        \step[fieldsource = eprint, final]
        \step[fieldset = url, null]
    }
  }
}

\usepackage[biblatex=true]{embrac} %

\usepackage{tikz}
\usetikzlibrary{cd} %

\usepackage[shortlabels,inline]{enumitem}
\setlist[enumerate,1]{label=\upshape(\arabic*)}
\newlist{myenumi}{enumerate}{1}
\setlist[myenumi,1]{label=\upshape(\roman*)}
\newlist{myenuma}{enumerate}{1}
\setlist[myenuma,1]{label=\upshape(\alph*)}

\usepackage{thmtools}
\declaretheorem[name=Theorem, numberwithin=section]{theorem}
\declaretheorem[name=Theorem, numbered=no]{theorem*}
\declaretheorem[name=Lemma, numberlike=theorem]{lemma}
\declaretheorem[name=Lemma, numbered=no]{lemma*}
\declaretheorem[name=Corollary, numberlike=theorem]{corollary}
\declaretheorem[name=Proposition, numberlike=theorem]{proposition}
\declaretheorem[name=Definition, numberlike=theorem, style=definition]{definition}
\declaretheorem[name=Convention, numberlike=theorem, style=definition]{convention}

\declaretheorem[name=Example, numberlike=theorem, style=remark]{example}
\declaretheorem[name=Remark, numberlike=theorem, style=remark]{remark}

\numberwithin{equation}{section}
\allowdisplaybreaks[1]

\usepackage[noabbrev,capitalise]{cleveref} %
\crefdefaultlabelformat{#2\textup{#1}#3}

\usepackage[disable]{todonotes}
\NewDocumentEnvironment{mytodoenv}{m}{\hypersetup{hidelinks}\textsf{\textbf{#1:}} }{}

\makeatletter
\providecommand\@dotsep{5}
\def\listtodoname{List of Todos}
\def\listoftodos{\@starttoc{tdo}\listtodoname}
\makeatother

\usepackage{mymacros}

\title[Rigidity of spin NNSC fill-ins]{Rigidity of spin fill-ins with non-negative scalar curvature} %

\author{Simone Cecchini}
\address[Simone Cecchini]{Department of Mathematics, Texas A\&M University}
\email{\href{mailto:cecchini@tamu.edu}{cecchini@tamu.edu}}
\urladdr{\href{https://simonececchini.org}{simonececchini.org}}

\author{Sven Hirsch}
\thanks{SH was supported by the National Science Foundation under Grant No. DMS-1926686, and by the IAS School of Mathematics.}
\address[Sven Hirsch]{Institute for Advanced Study, 1 Einstein Drive, Princeton, NJ 08540}
\email{\href{mailto:sven.hirsch@ias.edu}{sven.hirsch@ias.edu}}
\urladdr{\href{https://www.svenhirsch.com}{www.svenhirsch.com}}
\author{Rudolf Zeidler}
\thanks{RZ: Funded by the European Union (ERC Starting Grant 101116001 – COMSCAL). Views and opinions expressed are however those of the author(s) only and do not necessarily reflect those of the European Union or the European Research Council. Neither the European Union nor the granting authority can be held responsible for them.}
\thanks{RZ: Funded by the Deutsche Forschungsgemeinschaft (DFG, German Research Foundation) – Project numbers
523079177; %
427320536; %
390685587  %
}
\address[Rudolf Zeidler]{University of Münster, Mathematisches Institut, Einsteinstr.\ 62, 48149 Münster, Germany}
\curraddr{Universität Potsdam, Institut für Mathematik, Karl-Liebknecht-Str.\ 24--25, 14476 Potsdam, Germany}
\email{\href{mailto:rudolf.zeidler@uni-potsdam.de}{rudolf.zeidler@uni-potsdam.de}}
\urladdr{\href{https://www.rzeidler.eu}{www.rzeidler.eu}}

\hypersetup{
  pdfauthor={Simone Cecchini, Sven Hirsch, Rudolf Zeidler}, %
}
\begin{document}

\begin{abstract}
We establish new mean curvature rigidity theorems for spin fill-ins with non-negative scalar curvature using two different spinorial techniques.
Our results address two questions by Miao and Gromov, respectively.
The first technique is based on extending boundary spinors satisfying a generalized eigenvalue equation via the Fredholm alternative for an APS boundary value problem, while the second is a comparison result in the spirit of Llarull and Lott using index theory.
We also show that the latter implies a new Witten-type integral inequality for the mass of an asymptotically Schwarzschild manifold, which holds even when the scalar curvature is not assumed to be non-negative.
\end{abstract}
\maketitle

\section{Introduction}

Fill-ins and extensions with non-negative scalar curvature play a major role in mathematical relativity and geometry.

On the one hand, they provide an important tool in many proofs.
In \cite{bray2001proof, bunting1987nonexistence, cederbaum2017uniqueness, hirsch2020positive} manifolds with boundary are filled in by disks, in \cite{mccormick2019penrose} a collar fill-in is used to connect the boundary to a minimal surface, and in \cite{shi-tam-manifolds-with-boundary,wang2009isometric} asymptotically flat extensions of compact manifolds with boundary are constructed.
The common thread in all of these constructions is that the fill-ins and extensions allow the application of theorems which hold for manifolds without boundary or with minimal boundary such as the positive mass theorem or the Riemannian Penrose inequality.
The above methods can even be combined as in \cite{cederbaum2017uniqueness, mantoulidis2015bartnik}.
Moreover, fill-ins can be used to simplify the topology of the underlying manifold, see for instance \cite{agostiniani2023new, bray2022harmonic}.

On the other hand, fill-ins and extensions are used directly in several important definitions such as Bray's inner mass \cite{bray2001proof} and the Bartnik mass \cite{bartnik1989new} which is the subject of much current research.

In view of this plethora of applications, it is of great significance to better understand when fill-ins with non-negative scalar curvature exist and what their properties are.
In this article, we address two questions asked by Miao \cite[Question 2]{miao2021nonexistence} and Gromov \cite[page 3]{gromov2019scalar} concerning such fill-ins.

Let us start by recalling the notion of a fill-in with non-negative scalar curvature.
For a Riemannian manifold $M$ with boundary $\partial M$, we denote by $\mean_{\partial M}$ the mean curvature of $\partial M$.
We adopt the convention that the boundary of the unit disk $\Disk^n\subset\R^n$ has mean curvature \(n-1\).

\begin{definition}
Let $\Sigma$ be an $(n-1)$-dimensional closed manifold.
Given a metric $g_\Sigma$ on $\Sigma$ and a smooth function $h\colon\Sigma\to\R$, we say that an $n$-dimensional connected compact Riemannian manifold with boundary $(M,g)$ is a non-negative scalar curvature (NNSC) fill-in for $(\Sigma,g_\Sigma,h)$ if $(\partial M,g|_{\partial M})=(\Sigma,g_\Sigma)$, $\scal_M\geq0$ and $\mean_{\partial M}=h$.
We say that $h$ is the mean curvature of the fill-in $(M,g,h)$.
When $\Sigma$ is spin, we say that $(M,g)$ is a spin NNSC fill-in if $M$ is spin and induces the given spin structure on $\Sigma=\partial M$.
\end{definition}

Shi--Wang--Wei's extension theorem~\cite[Theorem~1]{shi2022total} establishes that, if $M$ is a compact manifold with boundary $\Sigma$ and $g_\Sigma$ is a Riemannian metric on $\Sigma$, then $g_\Sigma$ extends to a metric of positive scalar curvature on $M$. 
This shows that every closed null-bordant Riemannian manifold admits an NNSC fill-in.
However, their construction produces a fill-in with negative mean curvature.
\textcite[Question 2]{miao2021nonexistence} asked whether every closed null-bordant Riemannian manifold $(\Sigma,g_\Sigma)$ admits a NNSC fill-in with positive mean curvature.
Our first result answers Miao's question in the negative in the spin setting.

\begin{theorem}\label{Thm:Berger}
Let $\Sigma$ be a closed spin manifold of dimension $n \geq 3$ with $n \equiv 0,1,3$ or $7 \mod 8$.
Then there exists a metric $g_\Sigma$ on $\Sigma$ such that every spin NNSC fill-in of $(\Sigma,g_\Sigma)$ with non-negative mean curvature is Ricci-flat with minimal boundary.
Certain Berger spheres $(\Sphere^3,g)$ have this property.
\end{theorem}

\cref{Thm:Berger} relies on a general extension principle for spinors exhibited in \cref{extension_lemma} (similar to earlier ideas of \textcite{Raulot:RigidityCompact,Hijazi_etal}, see~\cref{remark:previous_results}), combined with classical existence results of harmonic spinors on these Berger spheres due to Hitchin~\cite{Hitchin} and, in the more general cases, due to further results of both \textcite{Hitchin} as well as \textcite{Baer_Harmonic_spinors}.
We remark that harmonic spinors can often be produced independently of the metric using the Atiyah--Singer index theorem, e.g.\ if \(\Ahat({\Sigma}) \neq 0\), but then \(\Sigma\) would not be null-bordant (see~\cite[Chapter III, \S 11, (11.21)]{LawsonMichelsohn}), so it would not make sense to talk about fill-ins in these situations.
The significance of \cref{extension_lemma} is that it does not rely on the index theorem and in \cref{Thm:Berger} we can thus treat special metrics on null-bordant manifolds such as the spheres \(\Sphere^{4k+3}\).

This method also shows that mean-convex spin domains admitting a parallel spinor are extremal with respect to spin fill-ins as in the next theorem.

\begin{theorem}\label{Thm:ricci_flat_extremality}
    Let \((N,g_N)\) be a spin manifold which admits a parallel spinor and has a compact mean-convex boundary \(\Sigma = \partial N\), that is, \(\mean_{\partial N} \geq 0\).
    Let \(g_\Sigma\) be the induced metric on \(\Sigma\) and let \(h \in \Ct^\infty(\Sigma, \R)\) with \(h \geq \mean_{\partial N}\).
    Then every spin NNSC fill-in of \((\Sigma, g_{\Sigma}, h)\) admits a parallel spinor and in this case \(h = \mean_{\partial N}\).
    Moreover, if \(h \not\equiv 0\), then the existence of a spin NNSC fill-in of \((\Sigma, g_{\Sigma}, h)\) implies that \(\Sigma\) is connected.
\end{theorem}

Note that, in general, it is important that a spin NNSC fill-in requires an extension of the given spin structure on \(\Sigma\) to the filled-in manifold \(M\), see also the discussion in \cref{remark:spin_fill-in} below.
However, in those cases where \(\Sigma\) only admits one spin structure---such as the main example \(\Sphere^3\) of \cref{Thm:Berger}---the theorem thus already precludes the existence of a fill-in with positive mean curvature admitting any spin structure.

Furthermore, we establish that the extension principle of \cref{extension_lemma} implies a certain comparison result in the spirit of \textcite{GoetteSemmelmann,Lott} without using index theory, see \cref{APS_lott}.

Next, we investigate a positive upper bound for a spin NNSC fill-in of a closed Riemannian spin manifold $\Sigma$ in terms of metric properties of $\Sigma$, which---in contrast to the previous results---is based on index theory.
To this end, let us start by recalling the following notion due to Gromov~\cite[Section~1]{gromov2019scalar},  \cite[Section 4.3]{Gromov_four_lectures}.

\begin{definition}\label{Def:hyperspherical radius}
    Let $(Y,g_Y)$ be a $k$-dimensional closed orientable Riemannian manifold. The \emph{hyperspherical radius} of $(Y,g_Y)$, denoted by $\rad_{\Sphere^k}(Y,g_Y)$, is the supremum of the numbers $R>0$ such that there exists a smooth $\tfrac1R$-Lipschitz map $f\colon (Y,g_Y)\to (\Sphere^k,g_{\Sphere^k})$ of non-zero degree.
\end{definition}

Gromov~\cite[Section 2]{gromov2019scalar} proved that if $(\Sigma^{n-1}, g_{\Sigma})$ is a closed Riemannian spin manifold, $h$ is a smooth function on $\Sigma$, and $M$ is a spin fill-in for $(\Sigma,g_\Sigma,h)$, then $\min_{\Sigma} h \leq (n-1)/R$, where $R$ is the hyperspherical radius of $\Sigma$.
Moreover, he conjectured~\cite[Page~3]{gromov2019scalar} that disks are rigid, that is, equality is achieved if and only if $M$ is a disk in Euclidean space.
Our second main result establishes this conjecture.

\begin{theorem}\label{Thm:Gromov rigidity}
For $n\geq2$, let $(\Sigma,g_\Sigma)$ be an $(n-1)$-dimensional closed connected Riemannian spin manifold.
Let $h\colon \Sigma\to\R$ be a smooth function, and let $(M,g)$ be a spin NNSC fill-in of $(\Sigma,g_\Sigma,h)$.
Then
\begin{equation}\label{eq:spherical_radius}
    \min_{p\in\Sigma} h(p)\leq\frac{n-1}{\rad_{\Sphere^{n-1}}(\Sigma,g_\Sigma)}.
\end{equation}
Furthermore, equality in \labelcref{eq:spherical_radius} is achieved if and only if $(M,g)$ is the round disk in Euclidean space of radius \(R = \rad_{\Sphere^{n-1}}(\Sigma,g_\Sigma)\) and $h=\frac{n-1}{R}$.
\end{theorem}
One wishes to apply a scalar- and mean curvature rigidity result in the spirit of \textcite{GoetteSemmelmann,Lott}, for flatness, see specifically \textcite[Theorem 3.2]{wang2023gromovs}, to a map \(\Sigma \to \Sphere^{n-1}\) realizing the hyperspherical radius.
However, due to its definition in terms of a supremum, \emph{a priori} we do not have a \emph{smooth} map precisely realizing the hyperspherical radius.
We solve this issue by establishing the following almost rigidity statement for maps to a fixed bounded strictly convex smooth domain in Euclidean space from a fixed spin NNSC manifold with boundary, which is also of independent interest.

\begin{theorem}\label{Thm:Llarull intro}
    For $n\geq3$, let \(\Omega \subseteq \R^n\) be a bounded domain with smooth strictly convex boundary.
    Let \(M\) be a connected compact Riemannian spin manifold with boundary such that \(\scal_M \geq 0\). 
    Fix \(p \in [1,\infty)\).
    Then for every \(\varepsilon > 0\) there exists a \(\delta=\delta(M,\Omega,p,\varepsilon) > 0\) such that the following holds:
    
    For every smooth map \(f \colon \partial M \to \partial \Omega\) satisfying
    \begin{itemize}
        \item \(\Lip(f) \leq 1+\delta\),
        \item \(\mean_{\partial M} \geq \mean_{\partial \Omega} \circ f - \delta\),
        \item \(\deg(f) \neq 0\),
    \end{itemize}
    there exists an isometry \(\phi \colon \partial M \to \partial \Omega\) with \(\II_{\partial M} =  \phi^\ast \II_{\partial \Omega}\) such that \(f\) is \(\varepsilon\)-close to \(\phi\) in \(\SobolevW^{1,p}(\partial M, \R^n)\).
    Moreover, in this case \(M\) is flat and isometric to \(\Omega\).
\end{theorem}

The proof of this theorem is based on carefully analyzing the fundamental spinorial integral inequality established in \cref{main_integral_inequality}, which can be interpreted as a version of a result of Lott \cite[Theorem 1.1]{Lott} \enquote{with coefficients} in the spirit of Listing \cite{listing2010scalar}, a special case of which is also used by Brendle in \cite{brendle2024scalar}.
The existence of spinors used in this theorem is provided by index theory which in this case works regardless of the dimension's parity (see \cref{appendix_index_theorem}).

Another possible approach to \cref{Thm:Gromov rigidity} would be to first extract a convergent subsequence from a sequence of maps almost realizing the hyperspherical radius and then prove rigidity directly for the limiting Lipschitz map.
Then the difficulty would lie in the fact that this map is \emph{a priori} not necessarily smooth.
Such an alternative approach to \cref{Thm:Gromov rigidity} was developed, after the initial appearance of the present manuscript, by \textcite{baer2024diraceigenvalueshypersphericalradius}, who formulated the result in terms of Dirac eigenvalues.

In the situation of Llarull's theorem, a related low regularity Lipschitz rigidity result was recently established by \textcite{cecchini2023lipschitz}, where also singular metrics of $\SobolevW^{1,p}$-regularity ($p>n$) are considered.
Also note that \cref{Thm:Llarull intro} immediately implies a rigidity result for low-regularity Lipschitz maps with smooth metrics as stated in the next corollary.
However, \cref{Thm:Llarull intro} provides stronger control than merely a Lipschitz rigidity result because using an a priori compactness argument only yields \(\Ct^{0,\alpha}\)-subconvergence to a Lipschitz map rather than in \(\SobolevW^{1,p}\).
\begin{corollary}\label{cor:Lipschitz}
    For $n\geq3$, let \(\Omega \subseteq \R^n\) be a bounded, strictly convex domain with smooth boundary.
    Let \(M\) be a connected compact Riemannian spin manifold with boundary such that \(\scal_M \geq 0\). 
    If $f\colon\partial M\to\partial\Omega$ is a \embparen{not necessarily smooth} 1-Lipschitz map of non-zero degree such that \(\mean_{\partial M} \geq \mean_{\partial \Omega} \circ f\), then $f$ is an isometry and $M$ is flat and isometric to $\Omega$.
\end{corollary}
In the case when the map $f$ in  \cref{cor:Lipschitz} is \emph{smooth}, its statement already follows from results of \textcite[Theorem 3.2]{wang2023gromovs}, see also \cite[Theorem~1.1]{wang2022rigidity}.

Finally, we show in \cref{S:PMT} that the same techniques---that is, an index-theoretic existence result of a spinor and the integral inequality from \cref{main_integral_inequality}---also imply a Witten-type integral formula \cite{witten1981new} for the ADM-mass of an asymptotically Schwarzschild manifold.
Interestingly, unlike Witten's approach, this formula does not use non-negative scalar curvature, and it is still valid when the underlying manifold is only asymptotically Schwarzschild up to first order instead of second order.
In particular, this gives yet another proof of the Riemannian positive mass theorem and thereby addresses another question of \textcite[3]{gromov2019scalar} on the relation between the theorems of Llarull, Goette--Semmelmann and Lott~\cite{Llarull,GoetteSemmelmann,Lott} and the positive mass theorem.

The paper is organized as follows.
In \cref{sec:notation}, we recall the basic notation and conventions on spinor bundles and Dirac operators.
In \cref{sec:nnsc_fillin_positive_mean}, we prove the extension result for spinors (\cref{extension_lemma}) and use it to deduce \cref{Thm:Berger,Thm:ricci_flat_extremality} and the comparison result à la Lott (\cref{APS_lott}).
\cref{sec:integral_formula} is dedicated to proving the fundamental integral inequality (\cref{main_integral_inequality}) needed for the remaining results.
In \cref{sec:almost rigidity}, we establish \cref{Thm:Llarull intro}, from which we derive \cref{Thm:Gromov rigidity} and \cref{cor:Lipschitz}.
Finally, in \cref{S:PMT}, we discuss the application to asymptotically flat manifolds.

\subsection*{Acknowledgements}
SC and RZ gratefully acknowledge the hospitality of the IAS, where this work was initiated.
The authors also thank Thomas Schick for helpful comments and his interest in this work.

\section{Notation and conventions}\label{sec:notation}
\subsection{The boundary Dirac operator and Weitzenböck formula}\label{subsec:boundary_dirac}
Let \(M\) be a Riemannian manifold and \(S \to M\) a Dirac bundle in the sense of Gromov--Lawson~\cites[\S 1]{GL1983}[Chapter II \S 5]{LawsonMichelsohn} (for instance the spinor bundle if \(M\) is a spin manifold).
We denote by \(\clm\) the Clifford multiplication and by \(\Dirac = \sum_{i=1}^n \clm(e_i)\nabla_{e_i} \) the Dirac operator on \(S\).
We will study boundary value problems in this context, for the general theory of which we refer to \textcite{Baer-Ballmann:Boundary-value-problems-first-order,Baer-Ballmann:Guide-Boundary-value-problems-Dirac}.
To this end, we denote the \emph{interior} unit normal field of \(\partial M\) by \(\nu_M\) and we denote by $\mean_{\partial M}$ the mean curvature with respect to $\nu_M$. 
Next, we fix the \emph{boundary Dirac operator}
\begin{equation}
  \BdDirac \coloneqq \sum_{i=1}^{n-1} \clm^{\partial}(e_i) \nabla^\partial_{e_i} = \frac{1}{2} \mean_{\partial M} - \clm(\nu_M) \Dirac - \nabla_{\nu_M} \colon \Ct^\infty(\partial M, S^\partial) \to \Ct^\infty(\partial M, S^\partial), \label{eq:boundary_dirac}
\end{equation}
where \(S^\partial = S|_{\partial M}\), \(\clm^\partial(\omega) = \clm(\omega)\clm(\nu)\), \(\nabla^\partial_\xi = \nabla_\xi + \frac{1}{2} \clm^\partial(\nabla_\xi \nu)\).
This operator \(\BdDirac\) can be used as a canonical adapted operator on the boundary, see \cite[Appendix~1]{Baer-Ballmann:Guide-Boundary-value-problems-Dirac} to study boundary value problems for \(\Dirac\).
Note that \(\BdDirac \clm(\nu_M) = - \clm(\nu_M) \BdDirac\).

We then have the integrated Bochner--Lichnerowicz--Weitzenböck formula for \(\Psi\in\Ct^\infty(M, S)\),
\begin{align}\label{eq:schroedinger_lichnerowicz}
    \| \Dirac \Psi \|_{\Lp^2(M)}^2 = \| \nabla \Psi \|_{\Lp^2(M)}^2 &+ \left(\Curv^S \Psi, \Psi\right)_{\Lp^2(M)} \\
    &+ \int_{\partial M} \frac{1}{2} \mean_{\partial M} |\Psi|^2 - \langle \BdDirac \Psi, \Psi \rangle \dS \nonumber
\end{align}
see \cite[Appendix 1, equation (27)]{Baer-Ballmann:Guide-Boundary-value-problems-Dirac}.
Here $\Curv^S$ is given by
\begin{equation*}
    \Curv^S = \sum_{i < j} \clm(e_i) \clm(e_j) \mathrm{R}^{S}_{e_i, e_j}.
\end{equation*}
In the case that \(M\) is spin and \(S\) its spinor bundle, we have \(\Curv^S = \frac{\scal}{4}\), in which case the formula \labelcref{eq:schroedinger_lichnerowicz} is called the integrated \emph{Schrödinger--Lichnerowicz formula}.

\subsection{Generalized APS boundary conditions}\label{subsec:APS}
We need to impose boundary conditions to work with the Dirac operator on manifolds with boundary.
Given any self-adjoint first-order operator \(A\) on the boundary adapted to \(\Dirac\) in the sense of \cite[\S 3.2]{Baer-Ballmann:Guide-Boundary-value-problems-Dirac}, which might not necessarily be the canonical adapted operator described in \cref{subsec:boundary_dirac} above, we can define the corresponding (generalized) \textcite{APS} (APS) boundary condition by imposing on sections \(\Psi \in \SobolevH^{1}(M, S)\) that 
\[
\chi_{[0,\infty)}(A)\left(\Psi|_{\partial M}\right) = 0,
\]
where \(\chi_{[0,\infty)}(A)\) is the \(\Lp^2\)-orthogonal projection onto the non-negative part of the spectrum of \(A\).
This is always an elliptic boundary condition~\cite[Example~4.21]{Baer-Ballmann:Guide-Boundary-value-problems-Dirac}.

In \cref{sec:nnsc_fillin_positive_mean} we will apply this to the case \(A = \BdDirac + h\), where \(h \in \Ct^\infty(\partial M, \R)\) is some smooth function on the boundary and \(\BdDirac\) is the canonical adapted operator.
The adjoint condition of \(\chi_{[0,\infty)}(\BdDirac + h)\left(\Psi|_{\partial M}\right) = 0\) is given by
\[
\chi_{(0,\infty)}(\BdDirac - h)\left(\Psi|_{\partial M}\right) = 0,
\]
compare~\cite[\S 4.3]{Baer-Ballmann:Guide-Boundary-value-problems-Dirac}, where \(\chi_{(0,\infty)}\) denotes the \(\Lp^2\)-orthogonal projection onto the positive part of the spectrum. 
Note the subtle but important difference between non-negative and positive parts of the spectrum as well as the change in sign in front of the function \(h\).

\subsection{Spin maps and local boundary conditions}\label{subsec:spin_maps}
In this subsection, we set up another boundary value problem, namely the one we need for the proof of~\cref{Thm:Gromov rigidity}.
This is a conceptual elaboration of the approach used by \textcite{Lott} that has also been used by several other authors in recent times \cite{brendle2023rigidity, brendle2024scalar, wang2023gromovs}. 

Let \(f \colon (M, \partial M) \to (N, \partial N)\) be a smooth \emph{spin map} between Riemannian manifolds with boundary of the same dimension \(n\).
Being a spin map means that \(w_1(\T M) = f^\ast w_1( \T N)\) and \(w_2(\T M) = f^\ast w_2( \T N)\), or equivalently that the bundle \(\T M \oplus f^\ast \T N\) endowed with the pseudo-Riemannian bundle metric \(g_M \oplus (-g_N)\) admits a spin structure.\footnote{If both \(M\) and \(N\) are oriented, this is by coincidence equivalent to \(\T M \oplus f^\ast \T N\) having a spin structure also with respect to a positive definite bundle metric.
But in general the pseudo-Riemannian perspective is the topologically appropriate point of view, compare also~\textcite[\S 1.2 and \S 4]{tony2024scalar}, and it avoids having to introduce factors of \(\sqrt{-1}\) at some places in our proofs.}
We say that a \emph{spin structure on \(f\)} is a spin structure on \((\T M \oplus f^\ast \T N, g_M \oplus (-g_N))\).

Recall that the Clifford algebra \(\Cl_{n,n} = \Cl(\R^n \oplus \R^n, (\delta, -\delta))\) has a canonical irreducible representation \(\mathbf{c} \colon \Cl_{n,n} \xrightarrow{\cong} \End(\bigwedge \R^n)\) which is generated by the standard Clifford actions on the exterior algebra
\begin{align}
    \textbf{c}(v,0) \alpha &\coloneqq \clm(v) \alpha \coloneqq v \wedge \alpha - \ins_v \alpha, \label{eq:differential_forms_clifford}\\
    \textbf{c}(0,v) \alpha &\coloneqq \clms(v) \alpha \coloneqq v \wedge \alpha + \ins_v \alpha.\nonumber
\end{align}
Let \(S\) be the complexified spinor bundle associated to a chosen spin structure on \(\T M \oplus f^\ast \T N\) and the canonical representation \(\textbf{c}\), that is, 
\[S = \left(\mathrm{P}_{\mathrm{K}_{n,n}}(\T M \oplus f^\ast \T N) \times_{\mathrm{K}_{n,n},\textbf{c}} \bigwedge  \R^n \right) \otimes \C,\]
where \(\mathrm{K}_{n,n} \subset \Spin_{n,n} \subset \Cl_{n,n}^\times\) denotes the maximal compact subgroup covering \(\mathrm{S}(\Orth_n \times \Orth_n) \subset \SO_{n,n}\).
Note that \(S\) admits two different Clifford actions, one for vector fields \(\xi \in \Vfds(M)\) on \(M\) which we denote by \(\clm(\xi) = \textbf{c}(\xi, 0) \in \End(S)\) and the other for vector fields \(\eta \in \Vfds(N)\) on \(N\) which we denote by \(\clms(\eta) = \textbf{c}(0, f^\ast \eta) \in \End(S)\).
By construction, the usual Clifford algebra relations hold which we spell out in the following for convenience of the reader.
\begin{gather*}
    \clm(\xi)^2 = - |\xi|^2, \quad \clms(\eta)^2 = |\eta|^2, \quad \clm(\xi) \clms(\eta) = -\clms(\eta) \clm(\xi),\\
    \clm(\xi)^\ast = -\clm(\xi), \quad \clms(\eta)^\ast = \clms(\eta).
\end{gather*}

With respect to the Clifford action of \(M\), the bundle \(S\) is a Dirac bundle in the sense of Gromov--Lawson and we have the Dirac operator \(\Dirac = \sum_{i=1}^n \clm(e_i) \nabla_{e_i}\).
In this case, the interior curvature term \(\Curv^S\) in the Bochner--Lichnerowicz--Weitzenböck formula \labelcref{eq:schroedinger_lichnerowicz} is explicitly given by \(\Curv^S = \frac{\scal}{4} + \Curv^N\), where
\begin{equation}
    \Curv^N = -\frac{1}{2}\sum_{i < j} \clm(e_i \wedge e_j) \clms\left(\mathrm{R}^{\T N}(\D{f}(e_i) \wedge \D{f}(e_j))\right). \label{eq:twisting_curvature}
\end{equation}

Again we need to impose suitable boundary conditions.
In addition to APS-type conditions, in this case there are specific local boundary conditions well-adapted to the situation at hand.
To describe these, let \(\nu_M\) and \(\nu_N\) denote the interior unit normals of the boundaries of \(M\) and \(N\), respectively, and let \(s \colon \partial M \to \{ \pm 1 \}\) be a choice of sign for each connected component of \(\partial M\).
Then we may impose the local boundary condition
\begin{equation}
\clm(\nu_M) \psi = s \clms(\nu_N) \psi \quad \text{on \(\partial M\).}   \label{eq:boundary_condition}
\end{equation}
Rewriting this as \(\chi(\psi) = s \psi \) with \(\chi = \clms(\nu_N) \clm(\nu_M)\) (and using \(\chi \clm(\nu_M) = -\clm(\nu_M) \chi\) as well as \cref{boundary_dirac_commutation} below), one verifies that this is a self-adjoint elliptic boundary condition.
Note that for \(s = 1\), this corresponds to the boundary condition used by Lott~\cite{Lott}.
Moreover, the bundle \(S\) inherits a \(\Z/2\)-grading \(S = S^+ \oplus S^-\) from the standard even/odd grading on \(\bigwedge \R^n\) with respect to which \(\Dirac\) is odd.
Since \(\chi\) preserves the \(\Z/2\)-grading, we can thus define the index 
\[
    \ind(\Dirac, s) = \dim \ker(\Dirac^+, s) - \dim \ker(\Dirac^-, s),   
\]
where 
\[\Dirac = \begin{pmatrix} 0 & \Dirac^- \\ \Dirac^+ & 0 \end{pmatrix}\]
with respect to the decomposition \(S = S^+ \oplus S^-\) and we take the kernel on smooth (or equivalently \(\SobolevH^1\)-) sections satisfying the boundary condition \labelcref{eq:boundary_condition}.
A similar setup was considered in \cite[\S 2--3]{scalar_mean_dirac}, compare also \cite[Appendix B]{BaerBrendleHankeWang}.

We consider two standard examples.
The first is that this reduces to differential forms in case \(M = N\).

\begin{example}\label{example:identity}
    Let \(M = N\) be the same Riemannian manifold and \(f = \id\).
    Then \(\T M \oplus f^\ast(\T M) = \T M \oplus \T M\) and this bundle admits a canonical spin structure with respect to the bundle metric \(g_M \oplus (-g_M)\).
    The exterior algebra \(\bigwedge \T{M}\) is itself a fiberwise irreducible bundle of \(\Cl(\T M \oplus \T M, g_M \oplus (-g_M))\)-modules, so we may just take \(S = \bigwedge \T{M} \otimes \C\) and, in light of \labelcref{eq:differential_forms_clifford}, \(\Dirac = d + d^\ast\) is the de Rham operator on complex-valued differential forms.
    By the same token, for \(s = 1\) the boundary condition \labelcref{eq:boundary_condition} then becomes \(\ins_{\nu_M} \alpha = 0\)---in other words, absolute boundary conditions.
    In particular, \(\ker(\Dirac, 1)\) is isomorphic to the de Rham cohomology of \(M\) and so \(\ind(\Dirac, 1) = \chi(M)\).
    Similarly, for \(s = -1\) it means \(\nu_M \wedge \alpha = 0\)---that is, relative boundary conditions.
    For a general sign choice \(s\), it would mean using absolute boundary conditions on some boundary components and relative on the others.    
\end{example}

\begin{convention}
In the remaining part of this paper, we will assume that every domain in Euclidean space is bounded and has smooth boundary.
\end{convention}
The second example is where \(N = \Omega \subset \R^n\) is a convex domain, which is the case we will use in the proof of \cref{Thm:Gromov rigidity}.

\begin{example}\label{example:domain}
    Let \(N = \Omega \subset \R^n\) be a convex domain.
    Then \(\T \Omega = \Omega \times \R^n\) is the trivial bundle and hence \(\T M \oplus f^\ast \T \Omega = \T M \oplus \R^n\).
    In this case, the condition of \(f\) being spin is equivalent to \(M\) admitting a spin structure.
    Moreover, given a spin structure on \(M\) with its principal \(\Spin_n\)-bundle, we then obtain
    \[
      S = \left( \mathrm{P}_{\Spin_n}(\T M) \times_{\Spin_n, \clm} \bigwedge \R^n \right) \otimes \C.
    \]
    Note that as a \(\Cl_{n,0}\)-module the exterior algebra \(\bigwedge \R^n\) is not irreducible and it splits into multiple irreducible submodules.
    For instance, if \(n\) is even, we can identify \(S\) with \(\ReducedSpinBdl_M \otimes \ReducedSpinBdl_{\R^n}\), where \(\ReducedSpinBdl\) refers to the usual complex spinor bundles.
    The latter being trivial, we can further observe that \(S \cong \bigoplus_{i=1}^m \ReducedSpinBdl_M\), where $m = \dim \ReducedSpinBdl_{\R^n} = 2^{\frac{n}{2}}$, so a section of \(S\) can be identified with a system of spinors on \(M\)---this is essentially the point of view taken by \textcite{brendle2024scalar}.
    However, we will not make use of this observation and instead work with the abstract construction of the bundle, which works regardless of the dimension's parity.
    
    We also note that the bundle \(S\) and its Dirac operator \(\Dirac\) do not depend on the map \(f\) or the domain \(\Omega\) at all, but the boundary condition \labelcref{eq:boundary_condition} does.
    We then have 
    \[\ind(\Dirac, 1) = \deg(f).\]
    In the even-dimensional case, this is a consequence of \textcite{Lott}, but again there is no reason to restrict to the even-dimensional case and we provide a direct argument of this index formula in \cref{appendix_index_theorem} independently of the dimension's parity, see \cref{the_index_theorem}.
\end{example}

\section{NNSC fill-ins with non-negative mean curvature}
\label{sec:nnsc_fillin_positive_mean}

In this section, we answer Miao's question \cite[Question 2]{miao2021nonexistence} restricted to spin fill-ins, establish fill-in extremality of spin domains admitting a parallel spinor, and establish a comparison result in the spirit of \textcite{Lott}.
All of these results rely on the following general extension lemma. 

\begin{lemma}\label{extension_lemma}
    Let \(M\) be a compact connected Riemannian manifold with boundary endowed with a Dirac bundle \(S \to M\) such that \(\Curv^S \geq 0\) on \(M\). 
    Let \(h \colon \partial M \to [0,\infty)\) be a smooth non-negative function and let \(\psi_0 \in \Ct^\infty(\partial M, S)\) be a non-zero section satisfying \(\BdDirac \psi_0 = \frac{1}{2} h \psi_0\).
    Furthermore, assume that \(\mean_{\partial M} \geq h\).
    Then the bundle \(S \to M\) always admits a non-zero parallel section \(\Psi\) satisfying \(\Curv^S\Psi = 0\), and \(\mean_{\partial M} = h\).
    Moreover, if \(\mean_{\partial M} \not\equiv 0\), then \(\psi_0\) itself extends to a parallel section \(\Psi\) of \(S \to M\) satisfying \(\Curv^S\Psi = 0\).
\end{lemma}

\begin{proof}
   
    We consider the Dirac operator \(\Dirac \colon \SobolevH^1(M,S; B) \to \Lp^2(M,S)\) subject to the boundary condition \(B\) given by the spectral projection \(\chi_{[0,\infty)}(\BdDirac - \frac{1}{2} h)\) as described in \cref{subsec:APS}.
    Then the adjoint boundary condition \(B^{\mathrm{ad}}\) is determined by the projection \(\chi_{(0,\infty)}(\BdDirac + \frac{1}{2} h)\).
    We now apply the Fredholm alternative.

  First consider the case that the operator \(\Dirac \colon \SobolevH^1(M,S; B) \to \Lp^2(M,S)\) is surjective.
  Then choose a section \(\Psi_0 \in \Ct^\infty(M,S)\) such that \(\Psi_0|_{\partial M} = \psi_0\).
 By surjectivity, we can furthermore choose \(\Psi_1\) such that \(\Dirac \Psi_1 = - \Dirac \Psi_0\) and \(\chi_{[0,\infty)}(\BdDirac - \frac{1}{2} h)(\psi_1) = 0\) with \(\psi_1 \coloneqq {\Psi_1}|_{\partial M}\).
 We then set \(\Psi = \Psi_0 + \Psi_1\). Then \(\Dirac \Psi = 0\) and hence by \eqref{eq:schroedinger_lichnerowicz} we have
 \begin{align*}
  0 &\geq \|\nabla \Psi\|^2_{\Lp^2} + \int_{\partial M} \langle (\tfrac{1}{2} \mean_{\partial M} - \BdDirac) \Psi, \Psi \rangle \dS \\
  &\geq \|\nabla \Psi\|^2_{\Lp^2} + \int_{\partial M} \langle (\tfrac{1}{2} h - \BdDirac) \Psi, \Psi \rangle \dS \\
  &= \|\nabla \Psi\|^2_{\Lp^2} + \int_{\partial M} \langle (\tfrac{1}{2} h - \BdDirac) \psi_1, \psi_1 \rangle \dS + \int_{\partial M} \langle (\tfrac{1}{2} h - \BdDirac) \psi_0, \psi_0 \rangle \dS \\
  &= \|\nabla \Psi\|^2_{\Lp^2} + \int_{\partial M} \langle (\tfrac{1}{2} h - \BdDirac) \psi_1, \psi_1 \rangle \dS \geq \|\nabla \Psi\|^2_{\Lp^2} + \lambda_0 \int_{\partial M} |\psi_1|^2 \dS \geq 0,
 \end{align*}
 Here the cross terms vanish because \(\frac{1}{2} h - \BdDirac\) is self-adjoint and \((\frac{1}{2} h - \BdDirac)\psi_0 = 0\), while \(\lambda_0 > 0\) denotes the smallest positive eigenvalue of the operator \(\tfrac{1}{2} h - \BdDirac\).
 This means that \(\nabla \Psi = 0\), \(\psi_1 = 0\) and hence \(\Psi|_{\partial M} = \psi_0\).
 So we have found the desired parallel section \(\Psi\) and it extends \(\psi_0\).
 Moreover, it then follows from \labelcref{eq:boundary_dirac} that \(\BdDirac \psi_0 = \frac{1}{2} \mean_{\partial M} \psi_0\) because \(\psi\) is nowhere vanishing as an extension of a parallel section.
 Since by assumption we have \(\BdDirac \psi_0 = \frac{1}{2} h \psi_0\) and since $\psi_0$ is nonvanishing almost everywhere, this means that \(h = \mean_{\partial M}\).
 
 On the other hand, if the operator \(\Dirac \colon \SobolevH^1(M,S; B) \to \Lp^2(M,S)\) is not surjective, then the operator \(\Dirac\), subject to the adjoint condition \(B^{\mathrm{ad}}\), must have a non-trivial element \(\Psi\) in the kernel, that is, \(\Dirac \Psi = 0\) and \(\chi_{(0,\infty)}(\BdDirac + \frac{1}{2} h)(\Psi|_{\partial M})\) = 0.
  By \eqref{eq:schroedinger_lichnerowicz},
  \begin{align*}
    0 = \|\Dirac \Psi\|_{\Lp^2(M)}^2 &\geq \| \nabla \Psi \|_{\Lp^2}^2 + \int_{\partial M} \langle (\tfrac{1}{2} \mean_{\partial M} - \BdDirac) \Psi, \Psi \rangle \dS \\
    &\geq \int_{\partial M} \tfrac{1}{2} \mean_{\partial M}  |\Psi|^2 \dS 
    + \int_{\partial M} \langle (-\tfrac{1}{2} h - \BdDirac) \Psi, \Psi \rangle \dS \geq 0,
  \end{align*}
  where we crucially use non-negativity of \(h\) and \(\mean_{\partial M}\).
  We conclude that \(\nabla \Psi = 0\), which provides the desired parallel section for this case, as well as \(0 = \mean_{\partial M} \geq h \geq 0\) and so also \(\mean_{\partial M} = h = 0\).
  In both cases, all non-negative terms in \labelcref{eq:schroedinger_lichnerowicz} vanish, so in particular \((\Curv^S \Psi,\Psi)_{\Lp^2(M)} = 0\).
  Since \(\Curv^S \geq 0\) pointwise on \(M\) as a fiberwise self-adjoint endomorphism, this implies \(\Curv^S \Psi = 0\).
  In particular, this second case can only occur if \(\mean_{\partial M} = 0\).
\end{proof}

Applying \cref{extension_lemma} to the spinor bundle itself yields the following statement.

\begin{proposition}\label{prop:spinor_extension}
    Let \((\Sigma^{n-1},g_\Sigma)\) be a closed spin manifold, and let \(h \colon \Sigma \to [0,\infty)\) be a smooth non-negative function such that there exists a non-trivial spinor \(\psi \in \Ct^\infty(\Sigma, \ReducedSpinBdl_\Sigma)\) satisfying \(\ReducedSpinDirac_\Sigma \psi = \frac{1}{2} h \psi\).
    Let \((M,g)\) be a connected compact Riemannian spin manifold with boundary such that \(\scal_M \geq 0\).
    Let \(\partial_0 M \subseteq \partial M\) be a union of connected components such that \((\partial_0 M,g|_{\partial_0 M})=(\Sigma,g_\Sigma)\) \embparen{with the given spin structure on \(\Sigma\)}.
    Assume that \(\mean_{\partial M} \geq h\) on \(\partial_0 M\) and \(\mean_{\partial M} \geq 0\) everywhere.
    Then \(M\) admits a parallel spinor and thus is Ricci-flat. Moreover, in this case \(\mean_{\partial M} = h\) on \(\partial_0 M\) and \(\mean_{\partial M}=0\) on \(\partial M \setminus \partial_0 M\).
    If, in addition, \(h \not\equiv 0\), then \(\partial_0 M\) is connected and \(\partial_0 M = \partial M\).
\end{proposition}
\begin{proof}
    Let \(S \to M\) denote the spinor bundle of \(M\) with Dirac operator \(\Dirac\).
    Then, along \(\partial_0 M\), the restriction \(S^\partial|_{\partial_0 M}\) can be identified with one or two copies (depending on the dimension parity) of the spinor bundle \(\ReducedSpinBdl_\Sigma\) such that the canonical boundary operator \(\BdDirac\) identifies with \(\ReducedSpinDirac_\Sigma\) on each copy.
    Thus, by assumption there exists a section \(\psi_0 \in \Ct^\infty(\partial_0 M, S^\partial)\) such that \(\BdDirac \psi_0 = \frac{1}{2} h \psi_0\) on \(\partial_0 M\).

    Extend \(\psi_0\) by \(0\) to a section \(\tilde\psi_0 \in \Ct^\infty(\partial M, S^\partial)\), and extend \(h\) by \(0\) to a smooth function \(\tilde h \colon \partial M \to [0,\infty)\).
    Then \(\BdDirac \tilde\psi_0 = \frac{1}{2} \tilde h \tilde\psi_0\) on \(\partial M\) and \(\mean_{\partial M} \geq \tilde h\).
    Hence the first part of the statement follows from \cref{extension_lemma}.
    In particular, \(\mean_{\partial M}=\tilde h\), so \(\mean_{\partial M}=h\) on \(\partial_0M\) and \(\mean_{\partial M}=0\) on \(\partial M\setminus\partial_0M\).

    Finally, assume \(h \not\equiv 0\).
    Hence \(\mean_{\partial M}\not\equiv 0\).
    Let \(\Gamma\) be a connected component of \(\partial_0M\) such that \(\psi_0|_{\Gamma}\not\equiv 0\), and let \(\tilde\psi_\Gamma\) be the extension of \(\psi_0|_{\Gamma}\) by \(0\) to all of \(\partial M\).
    Since \(\BdDirac\psi_0=\frac{1}{2}h\psi_0\) on \(\partial_0M\), \(\tilde h\) is the extension of \(h\) by \(0\) to \(\partial M\), and \(\tilde\psi_\Gamma=0\) on \(\partial M\setminus\Gamma\), we have
    \[
      \BdDirac \tilde\psi_\Gamma = \frac{1}{2}\tilde h\,\tilde\psi_\Gamma
      \quad\text{on }\partial M.
    \]
    Therefore, by the last statement of \cref{extension_lemma}, \(\tilde\psi_\Gamma\) extends to a non-zero parallel spinor on \(M\), which is nowhere vanishing.
    Since \(\tilde\psi_\Gamma\) vanishes on \(\partial M\setminus\Gamma\), this forces \(\partial M=\Gamma\).
    In particular, \(\partial_0M=\partial M\) and \(\partial_0M\) is connected.
\end{proof}

We are now ready to prove the first two of our main theorems.

\begin{proof}[Proof of \cref{Thm:Berger}]
    On an infinite (but non-generic) family of Berger metrics on \(\Sphere^3\), there are non-trivial solutions to the Dirac equation according to \cite[\S3.1]{Hitchin}.
    Similarly, Bär showed in \cite{Baer_Harmonic_spinors} that on every closed spin manifold of dimension $n\equiv 3 \mod 4$, there exists a metric admitting a harmonic spinor, and Hitchin did so in \cite[Theorem 4.5]{Hitchin} for dimension $n\equiv 0,1,7\mod 8$.
    Hence, the result follows from \cref{prop:spinor_extension} with \(h = 0\).
\end{proof}

\begin{proof}[Proof of \cref{Thm:ricci_flat_extremality}]
    Given a parallel spinor on \(N\), its restriction to \(\Sigma\) provides a spinor \(\psi_0\) satisfying \(\BdDirac \psi_0 = \frac{1}{2}\mean_{\partial N} \psi_0\) because of \labelcref{eq:boundary_dirac}.
    Let \((M,g)\) be a spin NNSC fill-in of \((\Sigma, g_\Sigma,h)\). Since \(\mean_{\partial M}=h \geq \mean_{\partial N}\geq0\), we can apply \cref{prop:spinor_extension} with \(\partial_0M=\partial M\) and \(h\) replaced by \(\mean_{\partial N}\).
    This gives a parallel spinor on \(M\) and \(\mean_{\partial M}=\mean_{\partial N}\), hence \(h=\mean_{\partial N}\).
    If \(h\not\equiv0\), then also \(\mean_{\partial N}\not\equiv0\), so the last part of \cref{prop:spinor_extension} yields connectedness of \(\partial M=\Sigma\).
\end{proof}

\begin{remark}\label{remark:extendability_issue}
    The proof of \cref{extension_lemma} has the curious feature that both cases of the Fredholm alternative lead to the existence of a spinor with the desired property.
    One might be tempted to ask more generally whether in the setting of the theorem any spinor on the boundary \(\psi_0 \in \Ct^\infty(\partial M, S)\) satisfying \(\BdDirac \psi_0 = \frac{1}{2} h \psi_0\) can be extended to a parallel spinor on a spin NNSC fill-in with \(\mean \geq h\).
    
    The proof already shows that this is the case whenever the first part of the Fredholm alternative applies.
    Moreover, the second case can only occur if \(h = 0\).
    So if \(h \neq 0\), any such spinor indeed extends to a parallel spinor on the fill-in.
    
    However, if \(h = 0\) and we have a non-trivial spinor \(\psi_0\) at the boundary satisfying \(\BdDirac \psi_0 = 0\), then \(\psi_1 = \clm(\nu_M) \psi_0\) also satisfies \(\BdDirac \psi_1 = 0\).
    But it is not possible to simultaneously extend both \(\psi_0\) and \(\psi_1\) to harmonic spinors on the spin fill-in \(M\), because if \(\Psi_i\) was a harmonic spinor on \(M\) extending \(\psi_i\) for \(i = 0,1\), then we would have
    \[
        0 = (\Dirac \Psi_0, \Psi_1)_{\Lp^2(M)} - ( \Psi_0, \Dirac\Psi_1)_{\Lp^2(M)} = \int_{\partial M} \langle \psi_0, \clm(\nu_M) \psi_1 \rangle \dS = - \int_{\partial M} \lvert \psi_0 \rvert^2 \dS \neq 0.
    \]
    
    Relatedly, non-negativity of \(h\) cannot be dropped from the hypotheses of \cref{extension_lemma} because if \(\BdDirac \psi_0 = \frac{1}{2} h \psi_0\), then we have \(\BdDirac \psi_1 = \frac{1}{2} (-h) \psi_1\) for \(\psi_1 = \clm(\nu_M) \psi_0\).
    For instance, the sphere \(\Sphere^{n-1}\) for \(n \geq 3\) admits a spinor \(\psi_1\) with \(\ReducedSpinDirac_{\Sphere^{n-1}} \psi_1 = -\frac{n-1}{2} \psi_1\) even though the hemisphere in \(\Sphere^n\) is a spin PSC fill-in of \(\Sphere^{n-1}\) with \(\mean = 0 \geq - \frac{n-1}{2}\).
    More drastically, the boundary Dirac operator always has arbitrarily negative eigenvalues.
\end{remark}

\begin{remark} \label{remark:previous_results}
A similar result as in \cref{prop:spinor_extension} was established by \textcite[Theorem~1]{Raulot:RigidityCompact} provided that \(h\) does not vanish identically.
However, for the rigidity statement in our main application, \cref{Thm:Berger}, we need to include the case \(h = 0\).

    A closely related statement was also claimed by \textcite[Theorem~6]{Hijazi_etal} formulated in terms of the smallest non-negative eigenvalue \(\lambda_1 \geq 0\) of the boundary Dirac operator \(\BdDirac\), which can be viewed as a special case of \cref{prop:spinor_extension} with constant \(h = 2 \lambda_1\).
    However, the statement and proof of \cite[Theorem~6]{Hijazi_etal} have a gap because it is claimed that if \(\frac{\mean_{\partial M}}{2} \geq \lambda_1\), any spinor satisfying \(\BdDirac \psi = \lambda_1 \psi\) on the boundary extends to a parallel spinor, which is not true in the case relevant for us, namely \(\lambda_1 = 0\), as we have discussed in \cref{remark:extendability_issue}.
    The root of the problem occurs in \cite[Theorem 2]{Hijazi_etal} which does not account for the possibility of having harmonic spinors on the boundary in which case the APS boundary value problem is not self-adjoint.
    
    Nevertheless, the first part of the statement in \cite[Theorem~6]{Hijazi_etal}, namely that for any NNSC fill-in \((M,g)\) we have \(
      \lambda_1 \geq \min_{\Sigma} \frac{\mean_{\partial M}}{2}
    \),
    remains true even if \(\lambda_1 = 0\), as it is a consequence of our \cref{extension_lemma}.\footnote{This fix is also exhibited in detail by \textcite[Appendix A]{baer2024diraceigenvalueshypersphericalradius}.}
    We also mention that another recent paper of \textcite{Raulot_DEC-fillin} treats a similar eigenvalue estimate in the setting of initial data sets and fill-ins satisfying the dominant energy condition.
\end{remark}

Next, we show that these methods imply a comparison result in the spirit of \textcite{GoetteSemmelmann,Lott}.
Indeed, the following theorem generalizes \cite[Theorem~1.3]{Lott} modulo the spin condition \labelcref{item:spin_condition}. 
Our proof does not use the Atiyah--Patodi--Singer index theorem as in \cite{Lott} but instead relies on the extension principle of \cref{extension_lemma}. 

\begin{theorem} \label{APS_lott}
    Let \(f \colon (M, \partial M) \to (N, \partial N)\) be a smooth spin map between connected Riemannian manifolds with boundary such that
    \begin{myenumi}
        \item the restricted map \(\partial f \coloneqq f|_{\partial M} \colon \partial M \to \partial N\) is an isometry, \label{item:boundary_isometry}
        \item the map \(f\) admits a spin structure which extends the canonical spin structure of \(\partial f\) from \cref{example:identity},\label{item:spin_condition}
        \item \(f \colon M \to N\) is area non-increasing, \label{item:area_condition}
        \item \(\scal_{M} \geq \scal_{N} \circ f\) and the curvature operator of \(N\) is non-negative, \label{item:scal_condition}
        \item \(\mean_{\partial M} \geq \mean_{\partial N} \circ \partial f\) and \(\mean_{\partial N} \geq 0\). \label{item:mean_condition}
    \end{myenumi}
    Then \(\scal_M = \scal_N \circ f\) and \(\mean_{\partial M} = \mean_{\partial N} \circ \partial f\).
    Moreover, if one of the two following additional conditions holds,
    \begin{itemize}
        \item \(0 <  \Ric_N < \frac{\scal_N}{2} \),
        \item or \(0 < \Ric_N\) and \(f\) is distance non-increasing,
    \end{itemize}
    then \(f \colon M \to N\) is an isometry.
\end{theorem}
\begin{remark}
    The condition on the spin structure \labelcref{item:spin_condition} is automatically satisfied if the restriction to the boundary induces a surjective map \(\HZ^1(N; \Z/2) \to \HZ^1(\partial N; \Z/2)\) on the first \(\Z/2\)-cohomology, for instance if \(\HZ^1(\partial N; \Z/2) = 0\).
    Indeed, the discrepancy between the induced boundary spin structure and the canonical one is measured by a class in \(\HZ^1(\partial M;\Z/2)\), which we may view as an element of \(\HZ^1(\partial N;\Z/2)\) via the isometry \(\partial f\).
    Twisting the spin structure on \(f\) by \(f^\ast\beta\) for \(\beta \in \HZ^1(N;\Z/2)\) changes this discrepancy by \(\beta|_{\partial N}\) (pulled back along \(\partial f\)).
    Thus, if \(\HZ^1(N;\Z/2) \to \HZ^1(\partial N;\Z/2)\) is surjective, we can choose \(\beta\) to cancel the discrepancy and obtain \labelcref{item:spin_condition}.
\end{remark}
\begin{proof}[Proof of \cref{APS_lott}]
    We work with the Dirac bundle \(S \to M\) as in \cref{subsec:spin_maps}, where we choose a spin structure on \(f\) with the property as in condition \labelcref{item:spin_condition}.
    We use the normal field \(\nu_M\) to identify \(\T M|_{\partial M} = \T(\partial M) \oplus \R\).
    Then \(S^\partial\) is a spinor bundle associated to a spin structure on the bundle \(\T (\partial M) \oplus \R \oplus (\partial f)^\ast \left( \T N|_{\partial N}\right)\) with respect to a suitable pseudo-Riemannian bundle metric (compare \cref{subsec:spin_maps}).
    The boundary connection \(\nabla^\partial\) on \(S^\partial\) as defined in \cref{subsec:boundary_dirac} then corresponds to the lift of the connection \(\nabla^{\partial M} \oplus \D \oplus f^\ast(\nabla^N)\), where \(\nabla^{\partial M}\) and \(\nabla^N\) are the Levi--Civita connections on \(\partial M\) and \(N\), respectively, and \(\D\) denotes the trivial connection.
    Next consider \(\bar{S} = \bigwedge \T N\) the exterior algebra on \(N\), which is also the spinor bundle associated to the canonical spin structure on \(\T N \oplus \T N\) as in \cref{example:identity}.
    Note that, along \(\partial N\), we can again identify \(\left(\T N \oplus \T N \right)|_{\partial N} = \T (\partial N) \oplus \R \oplus \T N\) using the normal field \(\nu_N\).
    The corresponding boundary connection \(\bar{\nabla}^\partial\) on \(\bar{S}\) as defined in \cref{subsec:boundary_dirac} again corresponds to the lift of the connection \(\nabla^{\partial N} \oplus \D \oplus \nabla^N\).
    Together with \labelcref{item:boundary_isometry}, we obtain an isometric identification
    \begin{align*}
        \left( \T M \oplus f^\ast \T N \right)|_{\partial M} &=  \T(\partial M) \oplus \R \oplus \left( f^\ast \T N \right)|_{\partial M} \\
        &\cong (\partial f)^\ast (\T (\partial N)) \oplus \R \oplus \left( f^\ast \T N \right)|_{\partial M} \\
        &= (\partial f)^\ast \left( \T (\partial N) \oplus \R \right) \oplus \left( f^\ast \T N \right)|_{\partial M} \\
        &= \left( f^\ast \left(  \T N \oplus \T N \right)\right)|_{\partial M}
    \end{align*}
    which takes the connection \(\nabla^{\partial M} \oplus \D \oplus f^\ast(\nabla^N)\) to \((\partial f)^\ast \left( \nabla^{\partial N} \oplus \D \oplus \nabla^N \right)\).
    Condition \labelcref{item:spin_condition} implies that this lifts to an identification \(S|_{\partial M} \cong (\partial f)^\ast (\bar{S}|_{\partial N}) = (\partial f)^\ast(\bigwedge \T N)\) taking \(\nabla^\partial\) to \(\bar{\nabla}^\partial\).
    In particular, the associated canonical boundary Dirac operators \(\BdDirac\) on \(S^\partial\) and \(\bar{\BdDirac}\) on \(\bar{S}^\partial\) agree along this isomorphism.
    
    Taking the preparation from the previous paragraph into account, now let \(\psi_0 = (\partial f)^\ast 1 \in \Ct^\infty(\partial M, S)\) be the section corresponding to the constant function \(1 \in \Ct^\infty(\partial N, \bigwedge^0 \T N)\) under this identification.
    Since this extends to the parallel section \(1 \in \Ct^\infty(N, \bigwedge^0 \T N)\) on \(N\), it follows by \labelcref{eq:boundary_dirac} that we have \(\bar{\BdDirac}(1) = \frac{\mean_{\partial N}}{2} 1\).
     Thus \(\BdDirac(\psi_0) = (\partial f)^\ast (\bar{\BdDirac} 1) = (\partial f)^\ast \left( \frac{1}{2}\mean_{\partial N} 1\right) = \frac{1}{2}(\mean_{\partial N} \circ (\partial f)) \psi_0\).
    Finally, it follows from the computation of \textcite[Lemma~1.1]{GoetteSemmelmann} together with conditions \labelcref{item:area_condition,item:scal_condition} that \(\Curv^S \geq 0\), compare also \cref{interior_curvature_estimate} below.
    Hence we are in the setting of \cref{extension_lemma} with \(h = \mean_{\partial N} \circ (\partial f)\) because of \labelcref{item:mean_condition}.
    Hence there exists a non-zero parallel section \(\Psi \in \Ct^\infty(M, S)\) satisfying \(\Curv^S \Psi = 0\), and \(\mean_{\partial M} = \mean_{\partial N} \circ (\partial f)\).
    The remaining statements then follow from the equality case analysis in \textcite{GoetteSemmelmann,Lott}.
\end{proof}

\begin{remark}\label{remark:spin_fill-in}
    The setup in the statement of \cref{prop:spinor_extension} also includes the spin structure in a subtle way.
    Indeed, we only obtain a result on fill-ins \(M\) that extend the particular spin structure of \(\Sigma\) on whose spinor bundle the given spinor \(\psi\) lives.
    That this is a necessary restriction can already be seen in the case of \(\Sphere^1\) which admits two spin structures, one of which extends to the disk but has no harmonic spinors, whereas the other admits a harmonic spinor but does not extend to the disc.
    In particular, the presence of the latter spin structure does not contradict the fact that the disk is an NNSC fill-in of \(\Sphere^1\) with positive mean curvature.
    Indeed, following \cref{subsec:boundary_dirac}, the spinor bundle of the disc restricted to the circle \(\Sphere^1\) can be identified with the trivial bundle \(\Sphere^1 \times \C^2\) with canonical boundary operator
    \[
      \BdDirac = \begin{pmatrix}
        -\iu \frac{\D}{\D \theta} + \frac{1}{2} & 0 \\
        0 & \iu \frac{\D}{\D \theta} + \frac{1}{2}
      \end{pmatrix}  
    \]
    which visibly does not have a kernel.
    A more intrinsic description is that sections of the spinor bundle on \(\Sphere^1\) associated to the spin structure restricted from the disc can be identified with \(2\pi\)-anti-periodic functions on \(\R\).
    The other spin structure on \(\Sphere^1\)---the one that does \emph{not} extend to the disc---corresponds to the trivial principal spin bundle and hence is just the trivial bundle \(\Sphere^1 \times \C\) on which the Dirac operator \(\iu \frac{\D}{\D \theta}\) acts. It clearly has a kernel given by the constant functions.
    See for instance \cite{Baer:Dependence_spin_structure} for more details on these and other examples of this nature.
    
    Of course, this aspect is not relevant in the case of the main example in \cref{Thm:Berger}, the Berger sphere, since \(\Sphere^3\) has only one spin structure.
    
    In the case of \cref{APS_lott}, there is also a spin extension condition as part of \labelcref{item:spin_condition}, but it is unclear if this is necessary.
\end{remark}

\section{The main integral formula}\label{sec:integral_formula}

In this section, we estimate the terms involving \(\Curv^N\) and \(\BdDirac\) from the Schr\"odinger--Lichnerowicz formula \labelcref{eq:schroedinger_lichnerowicz} in the setting of \cref{subsec:spin_maps} to obtain the integral inequality in \cref{main_integral_inequality}.
This combines the estimates of \textcite{GoetteSemmelmann,listing2010scalar,Lott}, but we need a more precise version of the boundary term for the proof of \cref{Thm:Llarull intro}.
We start with the following commutation relation.

\begin{lemma} \label{boundary_dirac_commutation}
    \begin{equation*}
        \BdDirac \clms(\nu_N) = \clms(\nu_N) \BdDirac + \clm(\nu_M) \sum_{i=1}^{n-1} \clm(e_i) \clms(- \nabla_{\D f(e_i)} \nu_N) 
    \end{equation*}
\end{lemma}
\begin{proof}
    Let \(\psi\) be a smooth section of \(S\) along \(\partial M\).
    First observe that \(\nabla^\partial_{\xi} \clms(\nu_N) \psi = \clms(\nabla_{\D{f}(\xi)} \nu_N) \psi + \clms(\nu_N) \nabla^\partial_\xi \psi \).
    We thus have
    \begin{align*}
        \BdDirac \clms(\nu_N) \psi &= \sum_{i=1}^{n-1} \clm(e_i) \clm(\nu_M) \nabla_{e_i}^\partial \clms(\nu_N) \psi \\
        &= \sum_{i=1}^{n-1} \left( \clm(e_i) \clm(\nu_M) \clms(\nabla_{\D{f}(e_i)} \nu_N) \psi + \clm(e_i)\clm(\nu_M)\clms(\nu_N) \nabla^\partial_{e_i} \psi \right) \\
        &= \left(\clm(\nu_M)  \sum_{i=1}^{n-1} \clm(e_i) \clms(- \nabla_{\D{f}(e_i)} \nu_N) \psi \right) + \clms(\nu_N) \BdDirac \psi. \qedhere
    \end{align*}
\end{proof}

To estimate the curvature terms further, we will use various norms of a linear map \(T \colon W \to V\) between Euclidean vector spaces \(W,V\), in particular the trace norm \(|T|_{\tr}\), the operator norm \(|T|_\op\), and the singular values \(\sigma_i(T)\), see \cref{appendix:matrix_norms} for details.

In the following, we will use the notation \(\Shape_{\partial N} = - \nabla \nu_N \colon \T(\partial N) \to \T(\partial N)\) to denote the Weingarten map of \(\partial N\).

\begin{lemma} \label{boundary_curvature_estimate}
    Suppose that \(\psi\) satisfies the boundary condition \labelcref{eq:boundary_condition}, that is, \(\clm(\nu_M) \psi = s \clms(\nu_N) \psi\).  
    Then pointwise on \(\partial M\) we have
    \begin{equation*}
        \langle \BdDirac \psi , \psi \rangle \leq \frac{1}{2} \lvert \Shape_{\partial N} \circ \D{f} \rvert_{\tr} \lvert \psi \rvert^2 - \frac{\sigma_{\min}(\Shape_{\partial N} \circ \D{f})}{4}\ \lvert \clms(U \blank) \psi - s \clm(\blank)\psi\rvert^2,
    \end{equation*}
    where \(U_x \colon \T_{x}(\partial M) \to \T_{f(x)}(\partial N)\) is an isometry coming from a polar decomposition of \(\Shape_{\partial N} \circ \D_x f \colon \T_{x}(\partial M) \to \T_{f(x)}(\partial N)\) and \(\sigma_{\min} = \sigma_{\min}(\Shape_{\partial N} \circ \D{f}) \geq 0\) denotes the smallest singular value of \(\Shape_{\partial N} \circ \D{f}\).
\end{lemma}
\begin{proof}
    It follows from the fact that $\BdDirac$ anticommutes with $\clm(\nu_M)$ and \(\clms(\nu_N) \clm(\nu_M)  \psi = s \psi\) that
    \begin{align*}
     \langle \BdDirac\psi,\psi\rangle=  & s \langle \BdDirac \psi, \clms(\nu_N) \clm(\nu_M)\psi \rangle= s\langle\clms(\nu_N)\BdDirac \psi, \clm(\nu_M)\psi \rangle\\
       =&s\langle (\clms(\nu_N)\BdDirac-\BdDirac\clms(\nu_N)) \psi, \clm(\nu_M)\psi \rangle - s\langle \BdDirac\clms(\nu_N) \clm(\nu_M)\psi,\psi\rangle\\
         =&s\langle (\clms(\nu_N)\BdDirac-\BdDirac\clms(\nu_N)) \psi, \clm(\nu_M)\psi \rangle - \langle \BdDirac\psi,\psi\rangle.
    \end{align*}
    Hence, \cref{boundary_dirac_commutation}  implies
    \begin{align*}
        \langle \BdDirac \psi, \psi \rangle &= -\frac{s}{2} \sum_{i=1}^{n-1}  \left\langle \clm(\nu_M) \clm(e_i)\clms( (\Shape_{\partial N} \circ \D{f})(e_i))\psi, \clm(\nu_M) \psi \right\rangle \\
        &= -\frac{s}{2} \sum_{i=1}^{n-1}  \left\langle \clm(e_i)\clms( (\Shape_{\partial N} \circ \D{f})(e_i))\psi, \psi \right\rangle.
    \end{align*}
    Now choose an orthonormal basis \((e_i)\) of \(\T_x(\partial M)\) such that \((\Shape_{\partial N} \circ \D{f})(e_i) = \sigma_i \bar{e}_i\), where \(\bar{e_i} = U e_i\) is an orthonormal basis of \(\T_{f(x)}(\partial N)\) and \(\sigma_i = \sigma_i(\Shape_{\partial N} \circ \D_x{f}) \geq 0\) are the singular values of \(\Shape_{\partial N} \circ \D_x f\).
    Then \(\lvert \Shape_{\partial N} \circ \D{f} \rvert_{\tr} = \sum_{i=1}^{n-1} \sigma_i\).
    We thus obtain
    \begin{align*}
         \langle \BdDirac \psi , \psi \rangle &= -\frac{s}{2} \sum_{i=1}^{n-1} \left\langle \clm(e_i)\clms( (\Shape_{\partial N} \circ \D{f})(e_i))\psi, \psi \right\rangle \\
        &= -\frac{s}{2} \sum_{i=1}^{n-1} 
        \sigma_i \underbrace{ \left\langle \clm(e_i)\clms(\bar{e}_i)\psi, \psi \right\rangle}_{-|\psi|^2 \leq \dots \leq |\psi|^2} \\
        &= \frac{1}{2} \lvert \Shape_{\partial N} \circ \D{f} \rvert_{\tr} \lvert \psi \rvert^2 -\frac{1}{2} \sum_{i=1}^{n-1} 
        \sigma_i \underbrace{\left( \lvert \psi \rvert^2 +  s \left\langle \clm(e_i)\clms(\bar{e}_i)\psi, \psi \right\rangle \right)}_{\geq 0} \\
        &\leq \frac{1}{2} \lvert \Shape_{\partial N} \circ \D{f} \rvert_{\tr} \lvert \psi \rvert^2 - \frac{\sigma_{\min}}{2} \sum_{i} \left(\lvert \psi \rvert^2 + s \left\langle \clm(e_i)\clms(\bar{e}_i)\psi, \psi \right\rangle \right) \\
        &= \frac{1}{2} \lvert \Shape_{\partial N} \circ \D{f} \rvert_{\tr} \lvert \psi \rvert^2 - \frac{\sigma_{\min}}{2} \sum_{i} \left(\lvert \psi \rvert^2 - s \left\langle \clms(\bar{e}_i)\psi, \clm(e_i) \psi \right\rangle \right)\\
        &= \frac{1}{2} \lvert \Shape_{\partial N} \circ \D{f} \rvert_{\tr} \lvert \psi \rvert^2 - \frac{\sigma_{\min}}{4} \sum_{i} \lvert \clms(U e_i) \psi - s \clm(e_i) \psi \rvert^2 \\
        &= \frac{1}{2} \lvert \Shape_{\partial N} \circ \D{f} \rvert_{\tr} \lvert \psi \rvert^2 - \frac{\sigma_{\min}}{4}\ \lvert \clms(U \blank) \psi - s \clm(\blank)\psi\rvert^2. \qedhere
    \end{align*}
\end{proof}

The same argument also shows:

\begin{lemma} \label{interior_curvature_estimate}
    For a section \(\Psi \in \Ct^\infty(M, S)\), pointwise on \(M\) we have
    \[
        \langle \Curv^N \Psi, \Psi \rangle  \geq -\frac{1}{2}\lvert \mathrm{R}^{\T N} \circ (\D{f} \wedge \D{f}) \rvert_{\tr} \lvert \Psi \rvert^2. %
    \]
 \end{lemma}
 \begin{proof}
    Choose an orthonormal basis \((\omega_\alpha)\) of \(\Lambda^2 \T_x M\) such that \(\mathrm{R}^{\T N}\circ (\D{f} \wedge \D{f})(\omega_\alpha) = \lambda_\alpha \bar{\omega}_\alpha\), where \((\bar{\omega}_\alpha)\) is an orthonormal basis of \(\Lambda^2 \T_{f(x)} N\) and \(\lambda_\alpha \geq 0\).
    Then \(\lvert \mathrm{R}^{\T N}\circ (\D{f} \wedge \D{f}) \rvert_{\tr} = \sum_{\alpha} \lambda_\alpha\).
    It follows from \labelcref{eq:twisting_curvature} that
    \[
        \langle \Curv^N \Psi, \Psi \rangle = -\frac{1}{2}\sum_{\alpha} \lambda_\alpha \langle \clm(\omega_\alpha) \clms(\bar{\omega}_\alpha) \Psi, \Psi \rangle
    \]
    and thus \(\langle \Curv^N \Psi, \Psi \rangle \geq -\frac{1}{2} \lvert \mathrm{R}^{\T N} \circ (\D{f} \wedge \D{f})\rvert_{\tr} \lvert \Psi \rvert^2\).
 \end{proof}

 Combining \cref{interior_curvature_estimate,boundary_curvature_estimate} we thus obtain the main integral inequalities:
 
 \begin{proposition}\label{main_integral_inequality}
    For any section \(\Psi \in \Ct^\infty(M, S)\) which satisfies the boundary condition \labelcref{eq:boundary_condition}, we have the following inequality.
    \begin{alignat*}{2}
        \| \Dirac \Psi \|_{\Lp^2(M)}^2 &\geq \| \nabla \Psi \|_{\Lp^2(M)}^2 &&+ \frac{1}{4} \left(\scal \Psi, \Psi\right)_{\Lp^2(M)} - \frac{1}{2}\int_M \lvert \mathrm{R}^{\T N} \circ \D{f} \wedge \D{f} \rvert_{\tr} \lvert \Psi \rvert^2 \dV \\
    & &&+ \frac{1}{2} \int_{\partial M}  \left(\mean_{\partial M} - \lvert \Shape_{\partial N} \circ \D{f}|_{\tr} \right) |\Psi|^2 \dS \\
    & &&+ \frac{1}{4} \int_{\partial M}  \sigma_{\min}(\Shape_{\partial N} \circ \D{f}) \lvert \clms(U \blank) \Psi - s \clm(\blank) \Psi\rvert^2 \dS,
    \end{alignat*}
    where \(U_x \colon \T_{x}(\partial M) \to \T_{f(x)}(\partial N)\) is a measurable section of bundle isometries coming from a polar decomposition of \(\Shape_{\partial N} \circ \D_x f \colon \T_{x}(\partial M) \to \T_{f(x)}(\partial N)\).
\end{proposition}
\Cref{main_integral_inequality} together with the matrix Hölder inequality \labelcref{eq:trace_hoelder} yields the estimate of \textcite{GoetteSemmelmann,listing2010scalar,Lott} because if the curvature operator of \(N\) is non-negative, then \(|\mathrm{R}^{\T N}|_{\tr} = \frac{\scal_N}{2}\), and if the boundary is convex, then \(|\Shape_{\partial N}|_{\tr} = \mean_{\partial N}\).
In the following corollary, we focus attention on the curvature term at the boundary.

	\begin{corollary} \label{boundary_rigidity}
	        In particular, if the boundary \(\partial N\) is convex and \(\lambda \colon \partial M \to (0,\infty)\) is a function with \(\lambda \geq \lvert \D{f} \rvert_{\op}\), then
	    \begin{alignat*}{2}
	        \| \Dirac \Psi \|_{\Lp^2(M)}^2 &\geq \| \nabla \Psi \|_{\Lp^2(M)}^2 &&+ \frac{1}{4} \left(\scal_M \Psi, \Psi\right)_{\Lp^2(M)} - \frac{1}{2}\int_M \lvert \mathrm{R}^{\T N} \circ \D{f} \wedge \D{f} \rvert_{\tr} \lvert \Psi \rvert^2 \dV \\
	    & &&+ \frac{1}{2} \int_{\partial M}  \left(\mean_{\partial M} - (\mean_{\partial N} \circ f) \cdot \lambda  \right) |\Psi|^2 \dS \\
	    & &&+ \frac{1}{4} \int_{\partial M}  \frac{(\sigma_{\min}(\Shape_{\partial N})\circ f)}{\lambda} \lvert \D{f} - \lambda U \rvert_2^2 |\Psi|^2 \dS \\
	    & &&+ \frac{1}{4} \int_{\partial M}  \sigma_{\min}(\Shape_{\partial N} \circ \D{f}) \lvert \clms(U \blank) \Psi - s \clm(\blank) \Psi\rvert^2 \dS,
	\end{alignat*}
	where \(U_x \colon \T_{x} \partial M \to \T_{f(x)}(\partial N)\) is a measurable section of bundle isometries coming from a polar decomposition \embparen{in the sense of Appendix~\labelcref{appendix:matrix_norms}} of the linear map \(\Shape_{\partial N} \circ \D_x f \colon \T_{x} \partial M \to \T_{f(x)}(\partial N)\).

	In particular, if in this situation \(\Dirac \Psi = 0\), \(\scal_M \geq  \lvert \mathrm{R}^{\T N} \circ \D{f} \wedge \D{f} \rvert_{\tr}\) and \(\mean _{\partial M} \geq (\mean_{\partial N} \circ f) \cdot \lambda\), then \(\nabla \Psi = 0\).
	If furthermore \(\partial N\) is strictly convex, then \(\lambda^{-1} \D{f} = U \colon \T (\partial M) \to \T(\partial N)\) is a bundle isometry and we have \(\clms(\D{f}(\xi)) \Psi = s \lambda \clm(\xi) \Psi\) on \(\partial M\) for any smooth vector field \(\xi\) tangent to \(\partial M\).
	 \end{corollary}
	 \begin{proof}
	     This is an immediate consequence of \cref{main_integral_inequality} and \cref{trace_inequality_polar_decomposition} from Appendix~\labelcref{appendix:matrix_norms}.
	     Indeed, since \(\partial N\) is convex, \(\Shape_{\partial N}\) is pointwise self-adjoint non-negative and \(|\Shape_{\partial N}|_{\tr} = \mean_{\partial N}\).
	     Applying \cref{trace_inequality_polar_decomposition} pointwise with \(S = (\Shape_{\partial N})_{f(x)}\), \(T = \D_x f\) and \(\lambda(x)\), we obtain an estimate for \(|\Shape_{\partial N} \circ \D f|_{\tr}\) in terms of \((\mean_{\partial N}\circ f)\cdot \lambda\) and \(\lvert \D f - \lambda U\rvert_2^2\), which yields the claimed inequality after inserting it into \cref{main_integral_inequality} and integrating over \(\partial M\).
	 \end{proof}
	 
	 In our main application, the curvature of \(N\) vanishes, so the interior curvature term \(\lvert \mathrm{R}^{\T N} \circ \D{f} \wedge \D{f} \rvert_{\tr}\) drops out completely. 
	 This is why we stop here to keep the notation reasonably light, but an analogous analysis can be applied to this interior term, including a more refined version of \cref{interior_curvature_estimate} that keeps some of the dropped terms in the same way as \cref{boundary_curvature_estimate} does for the boundary term.

\section{Almost rigidity for maps to Euclidean domains}\label{sec:almost rigidity}

In this section, we prove our main almost rigidity theorem for maps to Euclidean domains and deduce rigidity for Gromov's hyperspherical radius estimate.

\begin{proposition}\label{almost_rigidity_theorem}
    Let \(\Omega \subseteq \R^n\) be a bounded domain with smooth strictly convex boundary \(\partial \Omega\).
    Let \(M\) be a connected compact Riemannian spin manifold with boundary such that \(\scal_M \geq 0\). 
    Let \(\partial_0 M \subseteq \partial M\) be a connected component of the boundary.
    Let \(f_i \colon \partial_0 M \to \partial \Omega\) be a sequence of smooth maps such that, for each \(i \in \N\), we have that
    \begin{itemize}
        \item \(\operatorname{Lip}(f_i) \leq 1+\frac{1}{i}\),
        \item \(\mean_{\partial M} \geq 0\) and, if \(\partial M \neq \partial_0 M\), \(\mean_{\partial M}\) does not vanish identically on \(\partial M \setminus \partial_0 M\),
        \item \(\mean_{\partial M} \geq \mean_{\partial \Omega} \circ f_i - \frac{1}{i}\) on \(\partial_0 M\),
        \item \(\deg(f_i) \neq 0\).
    \end{itemize}
    Then \(\partial_0 M = \partial M\) and there exists a smooth isometry \(f \colon \partial M \to \partial \Omega\) such that \(\II_{\partial M} = f^\ast \II_{\partial \Omega}\) and a subsequence \((f_{i_k})_{k \in \N}\) such that \(f_{i_k} \to f\) in \(\SobolevW^{1,p}(\partial M, \R^n)\) for every \(p < \infty\).
 \end{proposition}

 \begin{proof}
    Fix \(y_0 \in \partial\Omega\). Define \(\tilde f_i \colon \partial M \to \partial\Omega\) by
    \[
        \tilde f_i|_{\partial_0 M} = f_i, \qquad \tilde f_i \equiv y_0 \text{ on every component of \(\partial M \setminus \partial_0 M\).}
    \]
    Extend the maps \(\tilde f_i\) to obtain maps \(F_i \colon (M, \partial M) \to (\Omega, \partial \Omega)\) of the same degrees and such that \(\D{F_i}(\nu_M) = \nu_\Omega\).
    Since \(\tilde f_i\) is constant on each component of \(\partial M \setminus \partial_0 M\), those components contribute degree \(0\), so \(\deg(F_i) = \deg(\tilde f_i) = \deg(f_i) \neq 0\).
    Since by \cref{example:domain} we have \(\ind(\Dirac, 1) = \deg(F_i)\), we find non-trivial harmonic spinors \(\Psi_i\) with \(\|\Psi_i\|_{\Lp^2(M)}^2 = 1\) satisfying the boundary condition \(\clm(\nu_M) \Psi_i = \clms_i(\nu_\Omega) \Psi_i\).
    Define \(\lambda_i \colon \partial M \to (0,\infty)\) by
    \[
        \lambda_i \coloneqq \begin{cases}
            1+\frac{1}{i}, & \text{on \(\partial_0 M\),}\\[2pt]
            \frac{1}{i}, & \text{on \(\partial M \setminus \partial_0 M\).}
        \end{cases}
    \]
    Then \(\lambda_i \geq \lvert \D{\tilde f_i} \rvert_{\op}\) on \(\partial M\): on \(\partial_0 M\) this is the Lipschitz assumption, while on \(\partial M \setminus \partial_0 M\) it follows from \(\D{\tilde f_i} = 0\).
    The estimate given by \cref{boundary_rigidity} (applied with \(\lambda = \lambda_i\) and \(s = 1\)) together with the fact that there exists a fixed constant \(\kappa > 0\), independent of \(i\), such that
    \[
        \mean_{\partial M} - (\mean_{\partial\Omega}\circ \tilde f_i)\lambda_i
        =
        \mean_{\partial M} - \frac{\mean_{\partial\Omega}(y_0)}{i}
        \geq
        \mean_{\partial M} - \frac{\kappa}{i}
        \quad \text{on \(\partial M\setminus\partial_0 M\)}
    \]
    and
    \[
        (\mean_{\partial\Omega} \circ f_i)\left(1+\frac{1}{i}\right) - \mean_{\partial M} \leq \frac{\kappa}{i}
        \quad \text{on \(\partial_0 M\),}
    \]
    shows that for all \(i\) we have
    \begin{align}
        \frac \kappa i \int_{\partial M}\lvert \Psi_i \rvert^2 \dS \geq \| &\nabla \Psi_i \|_{\Lp^2(M)}^2
        + \frac{1}{2}\int_{\partial M \setminus \partial_0 M}\mean_{\partial M}\lvert \Psi_i \rvert^2 \dS  \label{eq:applied_almost_rigid_estimate}\\
        &+ \|\lvert \D{f_i} - (1+\tfrac{1}{i}) U_i \rvert_2 |\Psi_i|\|_{\Lp^2(\partial_0 M)}^2 \nonumber \\
        &+ \left\| \sqrt{\sigma_{\min}(\Shape_{\partial\Omega} \circ \D{f_i})} \left( \clms(U_i \blank) \Psi_i - \clm(\blank)\Psi_i \right) \right\|_{\Lp^2(\partial_0 M)}^2. \nonumber
    \end{align}
    Here \(U_i\) denotes a measurable bundle isometry \(\T (\partial M) \to \tilde f_i^\ast \T (\partial \Omega)\) such that \(\Shape_{\partial\Omega} \circ \D{\tilde f_i} = P_i \circ U_i\) is a polar decomposition of \(\Shape_{\partial\Omega} \circ \D{\tilde f_i}\).
    Note that while \(U_i\) is uniquely determined and smooth on any subset where \(\D{\tilde f_i}\) is invertible because we have assumed \(\Shape_{\partial\Omega}\) to be strictly positive, we may have non-uniqueness and jumps in the presence of non-regular points of \(\tilde f_i\), which is why a priori we can only assume \(U_i\) to be a measurable section.
    
    By the trace inequality applied to the function $|\Psi_i|$, we also obtain a constant $C_{\mathrm{T}} \geq 0$ not depending on $i$ such that 
    \begin{align}
        \int_{\partial M}|\Psi_i|^2 \dS \le C_{\mathrm{T}}\|\Psi_i\|_{\SobolevW^{1,2}(M)}^2 = C_{\mathrm{T}} \left( 1 + \| \nabla \Psi_i \|^2_{\Lp^2(M)}\right) \label{eq:trace_estimate}
    \end{align}
    for all \(i\).
    Plugging \labelcref{eq:trace_estimate} into \labelcref{eq:applied_almost_rigid_estimate} shows that the right-hand side of \labelcref{eq:applied_almost_rigid_estimate} tends to zero as \(i \to \infty\), that is,
    
    \begin{align}
        \| \nabla \Psi_i \|_{\Lp^2(M)} &\to 0, \label{eq:nabla_to_zero}\\
        \|\lvert \D{f_i} - (1+\tfrac{1}{i}) U_i \rvert_2 |\Psi_i|\|_{\Lp^2(\partial_0 M)} &\to 0, \label{eq:differential_almost_iso} \\
        \| \sqrt{\sigma_{\min}(\Shape_{\partial\Omega} \circ \D{f_i})} \left( \clms(U_i \blank) \Psi_i - \clm(\blank)\Psi_i \right)\|_{\Lp^2(\partial_0 M)} &\to 0, \label{eq:clifford_almost_equal}\\
        \int_{\partial M\setminus\partial_0 M}\mean_{\partial M}|\Psi_i|^2 \dS &\to 0. \label{eq:mean_curvature_control}
    \end{align}
    
    In particular, the sequence \((\Psi_i)\) is uniformly bounded in \(\SobolevW^{1,2}(M)\), so after passing to a subsequence we may assume that \(\Psi_i \to \Psi\) in \(\Lp^2(M)\) for some spinor \(\Psi\) with \(\| \Psi \|_{\Lp^2(M)} = 1\) by the Rellich theorem.
    But then \labelcref{eq:nabla_to_zero} further implies that this convergence already takes place in \(\SobolevW^{1,2}(M)\) and \(\nabla \Psi = 0\), in particular \(\Psi\) is smooth on all of \(M\) and \(\lvert \Psi \rvert\) is constant.
     Moreover, by the trace theorem, 
    \begin{equation}
        \Psi_i|_{\partial M} \to \Psi|_{\partial M} \text{ in \(\Lp^2(\partial M)\).} \label{eq:boundary_convergence}
    \end{equation}
    Assume for contradiction that \(\partial M \setminus \partial_0 M \neq \emptyset\).
    By assumption, \(\mean_{\partial M} \geq 0\) on \(\partial M \setminus \partial_0 M\) and \(\mean_{\partial M}\) is not identically zero there, so
    \[
        \int_{\partial M \setminus \partial_0 M}\mean_{\partial M}\dS > 0.
    \]
    Since \(M\) is connected and \(\Psi\) is a non-trivial parallel spinor, \(|\Psi|\) is a positive constant on \(M\).
    Thus by \labelcref{eq:boundary_convergence} we obtain
    \[
        \int_{\partial M \setminus \partial_0 M}\mean_{\partial M}|\Psi_i|^2\dS
        \to
        \int_{\partial M \setminus \partial_0 M}\mean_{\partial M}|\Psi|^2\dS
        =
        |\Psi|^2\int_{\partial M \setminus \partial_0 M}\mean_{\partial M}\dS > 0,
    \]
    contradicting \labelcref{eq:mean_curvature_control}.
    Therefore \(\partial_0 M = \partial M\), and hence \(\tilde f_i = f_i\) on \(\partial M\).
    Since \(\D{f}_i\) and \(U_i\) are uniformly bounded in \(\Lp^\infty(\partial M)\), the convergence \labelcref{eq:boundary_convergence} together with \labelcref{eq:differential_almost_iso,eq:clifford_almost_equal} means that as \(i \to \infty\),
    \begin{align*}
        \|\D{f_i} - U_i \|_{\Lp^2(\partial M)} &\to 0, \\
        \| \sqrt{\sigma_{\min}(\Shape_{\partial\Omega} \circ \D{f_i})} \left( \clms(U_i \blank) \Psi - \clm(\blank)\Psi \right)\|_{\Lp^2(\partial M)} &\to 0, 
    \end{align*}
    respectively.
    After passing to a further subsequence, we can then ensure that pointwise almost everywhere on \(\partial M\) we have
    \begin{align}
        |\D{f_i} - U_i |^2 &\to 0, \label{eq:differential_close_to_isometry}\\
        \sigma_{\min}(\Shape_{\partial\Omega} \circ \D{f_i})\ | \clms(U_i \blank) \Psi - \clm(\blank)\Psi|^2 &\to 0. \nonumber
    \end{align}
    At a point \(x \in \partial M\) where \(|\D_x{f_i} - (U_i)_x | \to 0\), we have 
    \[\liminf_{i \to \infty} \sigma_{\min}((\Shape_{\partial\Omega} \circ \D{f_i})_x) > 0 \]
    because \(\liminf_{i \to \infty} \sigma_{\min}((\Shape_{\partial \Omega})_{f_i(x)}) > 0\) everywhere due to strict convexity of the boundary. 
    So we deduce that
    \begin{equation}
        |  \clms(U_i \blank) \Psi - \clm(\blank)\Psi | \to 0 \quad \text{pointwise almost everywhere on \(\partial M\).}\label{eq:clifford_action_convergence}
    \end{equation}
    Next, \labelcref{eq:clifford_action_convergence} implies that there exists a measurable section \(U\) of \(\Hom(\T (\partial M), \R^n)\) which is a fiberwise isometric embedding \(\T_x(\partial M) \hookrightarrow \R^n\) such that \(U_i \to U\) pointwise almost everywhere.
    This is because we have an injective bundle map 
    \begin{equation}
         \Hom(\T_x(\partial M) , \R^n) \hookrightarrow \Hom(\T_x(\partial M), S_x), \quad T \mapsto \clms(T \blank) \Psi.  \label{eq:bundle_embedding}
    \end{equation}
    Together with \labelcref{eq:differential_close_to_isometry} this implies \(\D{f}_i \to U\) pointwise almost everywhere on \(\partial M\).
    Again using that \(\D{f_i}\) and \(U\) are uniformly bounded in \(\Lp^\infty(\partial M)\), this implies that \(\D{f}_i \to U\) in \(\Lp^p(\partial M, \Hom(\T (\partial M) , \R^n))\) for every \(p < \infty\) by dominated convergence.
    
    Finally, we use Arzelà--Ascoli to pass to another subsequence such that \(f_i \to f\) in \(\Ct^0\), where \(f \colon \partial M \to \partial \Omega\) is a Lipschitz map.
    Since we already know that \(\D{f_i}\) converges in \(\Lp^p\), it follows that \((f_i)\) is a Cauchy sequence also in \(\SobolevW^{1,p}(\partial M, \R^n)\) and hence the convergence \(f_i \to f\) takes place in \(\SobolevW^{1,p}(\partial M, \R^n)\) for every \(p < \infty\).
    Furthermore, it follows that \(\D{f} = U\) almost everywhere and so \(\D{f}\) is an isometry almost everywhere. 
    
    Finally, we deduce from \labelcref{eq:clifford_action_convergence} that
    \begin{equation} 
        \clms(\D{f}(\xi)) \Psi = \clm(\xi) \Psi \label{eq:clifford_action}
    \end{equation}
    almost everywhere for any smooth vector field \(\xi\) tangent to \(\partial M\).
    Now since the parallel spinor \(\Psi\) is smooth, this means already that \(\D{f} \colon \T (\partial M) \to \R^n\) is smooth, again because of the smooth bundle embedding \labelcref{eq:bundle_embedding}.
    Thus \(f\) is smooth on \(\partial M\).
    But then \(f \colon \partial M \to \partial \Omega\) is a smooth local isometry and hence an isometry because \(\partial \Omega\) is diffeomorphic to the sphere.
    
    To obtain the result on the second fundamental forms, first observe that \(\Psi\) also satisfies the boundary condition \(\clms(\nu_\Omega) \Psi = \clm(\nu_M)\Psi\), where \(\clms(\nu_\Omega)\) is defined in terms of the limiting map \(f\).
    Differentiating this equality in the direction of a vector field \(\xi\) tangential to \(\partial M\), we obtain 
    \[
        \clm(\nabla_{\xi} \nu_M) \Psi = \clms(\nabla_{\D{f}(\xi)} \nu_\Omega) \Psi  \underset{\labelcref{eq:clifford_action}}{=} \clm((\D{f})^{-1}\nabla_{\D{f}(\xi)} \nu_\Omega) \Psi,  
    \]
    and thus \(\nabla_{\xi} \nu_M = (\D{f})^{-1}\nabla_{\D{f}(\xi)} \nu_\Omega\) which proves the desired statement \(\II_{\partial M} = f^{\ast} \II_{\partial \Omega}\).
 \end{proof}

The next proposition already follows from \cite[Theorem 3.2]{wang2023gromovs}.
We include a self-contained proof, which does not make use of the analysis on manifolds with corners for the odd-dimensional case.
\begin{proposition}\label{cor:flat}
    Let \(\Omega \subseteq \R^n\) be a bounded domain with smooth strictly convex boundary \(\partial \Omega\).
    Let \(M\) be a connected compact Riemannian spin manifold with boundary such that \(\scal_M \geq 0\).
    Let $f\colon\partial M\to\partial\Omega$ be an isometry satisfying \(\II_{\partial M} = f^\ast \II_{\partial \Omega}\).
    Then $M$ is isometric to $\Omega$.
\end{proposition}

\begin{proof}
    We work in the same setup as the proof of \cref{almost_rigidity_theorem} above.
    That is, extend $f$ to a smooth map \(F\colon (M, \partial M) \to (\Omega, \partial \Omega)\) of the same degree such that \(\D{F}(\nu_M) = \nu_\Omega\).    
    Since by \cref{example:domain} \(\ind(\Dirac, 1) = \deg(F) = \deg(f) \neq 0\), we find a non-trivial harmonic spinor \(\Psi\) satisfying the boundary condition \(\clm(\nu_M) \Psi = \clms(\nu_\Omega) \Psi\).
    By \cref{boundary_rigidity}, $\Psi$ is parallel and it satisfies 
    \begin{equation}
        \clm(\xi) \Psi = \clms(\D{F} (\xi)) \Psi \quad \text{on \(\partial M\) for any \(\xi \in \Ct^\infty(\partial M, \T M)\),} \label{eq:clifford_action_total_rigidity}
    \end{equation}
    Compare \labelcref{eq:clifford_action} above; for the normal component \(\xi = \nu_M\) this is just the boundary condition.

    Let \(\bar{e}_1, \dotsc, \bar{e}_n\) denote a constant orthonormal basis of \(\R^n\).
    For each \(i = 1, \dotsc, n\) define a vector field \(\xi_i \in \Vfds(M)\) by \(\langle \xi_i, \blank \rangle = \langle \clms(\bar{e}_i) \Psi, \clm(\blank) \Psi \rangle\).
    Since \(\Psi\) and \(\bar{e}_i\) are parallel, the vector field \(\xi_i\) is also parallel.
    Using \labelcref{eq:clifford_action_total_rigidity} on \(\partial M\) we have
    \begin{align*}
        \langle \xi_i, \xi \rangle = \langle \clms(\bar{e}_i) \Psi, \clms(\D{F}(\xi)) \Psi \rangle &= -\langle \clms(\bar{e}_i) \clms(\D{F}(\xi)) \Psi, \Psi \rangle + 2 \langle \bar{e}_i, \D{F}(\xi) \rangle \lvert \Psi \rvert^2 \\
        &= -\langle \clms(\bar{e}_i) \clm(\xi) \Psi, \Psi \rangle + 2 \langle \bar{e}_i, \D{F}(\xi) \rangle \lvert \Psi \rvert^2\\
        &= -\langle \xi_i, \xi \rangle + 2 \langle \bar{e}_i, \D{F}(\xi) \rangle \lvert \Psi \rvert^2,
    \end{align*}
    where we have used the Clifford relation \(\clms(v)\clms(w) + \clms(w)\clms(v) = 2 \langle v, w \rangle\).
    Thus
    \begin{equation}
        \langle \xi_i, \xi \rangle = \langle \bar{e}_i, \D{F}(\xi) \rangle \lvert \Psi \rvert^2 \quad \text{on \(\partial M\) for any \(\xi \in \Ct^\infty(\partial M, \T M)\).} \label{eq:vf_boundary_relation}
    \end{equation}
    In particular, assuming the parallel spinor \(\Psi\) is normalized such that \(|\Psi| = 1\), then \labelcref{eq:vf_boundary_relation} implies that \(\D{F}(\xi_i) = \bar{e}_i\) on \(\partial M\).
    Since \(\D{F}\) is an isometry at the boundary, this means that \(\xi_1, \dotsc, \xi_n\) is an orthonormal frame of \(\T M\) along \(\partial M\).
    But the vector fields \(\xi_i\) are parallel, so they form a parallel orthonormal frame of \(\T M\) on all of \(M\). This means that the metric on \(M\) is flat.

    Using that \(\II_{\partial M} = f^\ast \II_{\partial\Omega}\), we glue $\R^n\setminus\Omega$ to $M$, obtaining an open manifold equipped with a complete flat $\Ct^2$-metric (see e.g.~\cite[Lemma~4.1]{shi-tam-manifolds-with-boundary}).
    By the classification theorem for flat manifolds, we deduce that $M$ is a domain in $\R^n$.
    Using again that \(\II_{\partial M} = f^\ast \II_{\partial \Omega}\), we deduce that $M$ and $\Omega$ are related via a Euclidean motion by the fundamental theorem of hypersurfaces.
\end{proof}

\begin{proof}[Proof of \cref{Thm:Llarull intro}]
    Assume by contradiction that for some \(p \in [1,\infty)\) and \(\varepsilon_0 > 0\), we have that for every \(\delta > 0\) there exists a map \(f \colon \partial M \to \partial \Omega\) as in the statement of the theorem but which is at least \(\varepsilon_0\)-far away in \(\SobolevW^{1,p}(\partial M, \R^n)\) from any isometry \(\phi \colon \partial M \to \partial \Omega\) satisfying \(\phi^\ast \II_{\partial \Omega} = \II_{\partial M}\).
    Letting \(\delta = \frac{1}{i}\), this gives a sequence \(f_i \colon \partial M \to \partial \Omega\).
    Since \(\deg(f_i)\neq 0\), for each \(i\) there exists a connected component \(\partial_i M \subseteq \partial M\) such that \(\deg(f_i|_{\partial_i M}) \neq 0\).
    As \(\partial M\) has finitely many connected components, after passing to a subsequence we may assume \(\partial_i M = \partial_0 M\) for all \(i\).
    Set \(m_\Omega \coloneqq \min_{\partial\Omega}\mean_{\partial\Omega} > 0\), where positivity follows from strict convexity of \(\partial\Omega\).
    Since \(\mean_{\partial M} \geq \mean_{\partial\Omega}\circ f_i - \frac{1}{i}\) on \(\partial M\), for all sufficiently large \(i\) we have
    \[
        \mean_{\partial M} \geq m_\Omega - \frac{1}{i} \geq \frac{m_\Omega}{2} > 0 \quad \text{on \(\partial M\),}
    \]
    and \(\mean_{\partial M}\) cannot vanish identically on \(\partial M\setminus\partial_0 M\) if this set is nonempty.
    Hence, after discarding finitely many terms, the restricted sequence \(f_i|_{\partial_0 M}\) satisfies all assumptions of \cref{almost_rigidity_theorem}.
    Applying \cref{almost_rigidity_theorem}, after passing to a subsequence, we obtain \(\partial_0 M = \partial M\), that is, \(\partial M\) is connected.
    Moreover, we obtain an isometry \(\phi \colon \partial M \to \partial \Omega\) with \(\phi^\ast \II_{\partial \Omega} = \II_{\partial M}\) such that \(f_i \to \phi\) in \(\SobolevW^{1,p}(\partial M,\R^n)\), contradicting the choice of \(f_i\).
    Hence the first part of the theorem is established.
    
    But if we have an isometry \(\phi \colon \partial M \to \partial \Omega\) on the boundary satisfying \(\phi^\ast \II_{\partial \Omega} = \II_{\partial M}\), then \(M\) and \(\Omega\) are also isometric by \cref{cor:flat}.
\end{proof}

\begin{proof}[Proof of \cref{Thm:Gromov rigidity}]
Equivalently, we show that, if $h\geq(n-1)/R$, then $M$ is isometric to the disc $\Disk_R^n$ of radius $R$.
Since $\Sigma$ has hyperspherical radius $R$, for every \(\delta > 0\) there exists a smooth map $f \colon\partial M\to\Sphere^n_R$ such that \(\operatorname{Lip}(f) \leq 1 + \delta\) and \(\deg(f) \neq 0\).
Since $M$ is a fill-in, \(\mean_{\partial M}\geq (n-1)/R=\mean_{\partial\Disk^n_R} \circ f\).
By \cref{Thm:Llarull intro}, $M$ and $\Disk^n_R$ are isometric in case $n\ge3$. The case $n=2$ follows from Gauß--Bonnet. 
\end{proof}

\begin{proof}[Proof of \cref{cor:Lipschitz}]
    For any \(\delta > 0\), we can \(\Ct^0\)-approximate \(f\) by a smooth map \(f_\delta \colon \partial M \to \partial \Omega\) such that \(\Lip(f_\delta) \leq 1 + \delta\).\footnote{Using mollification in small coordinate patches on \(\partial M\) we can obtain such an approximation as a map \(\tilde{f}_\delta \colon \partial M \to \R^n\). Projecting back to \(\partial \Omega\) along a small tubular neighborhood yields the desired approximation \(f_\delta \colon \partial M \to \partial \Omega\).}
    By uniform continuity of \(\mean_{\partial \Omega}\), we can further assume that \(\mean_{\partial M} \geq \mean_{\partial \Omega} \circ f_\delta - \delta\).
    Fix \(p > n\).
    Then \cref{Thm:Llarull intro} proves that \(M\) is isometric to \(\Omega\) and that \(f_\delta\) can be assumed to be arbitrarily \(\SobolevW^{1,p}\)-close to an isometry.
    Hence \(f_\delta\) is \(\Ct^0\)-close to an isometry.
    In sum, we can \(\Ct^0\)-approximate the original \(1\)-Lipschitz map \(f \colon \partial M \to \partial \Omega \) arbitrarily well by isometries \(\partial M \to \partial \Omega\) and so \(f\) itself must be an isometry because metric isometries are closed in \(\Ct^0\).
\end{proof}

\section{Applications to asymptotically flat manifolds}\label{S:PMT}

Witten~\cite{witten1981new} showed that, on an asymptotically flat spin manifold $M^n$ of non-negative scalar curvature, the following identity holds:
\begin{align*}
    m=\frac{1}{2(n-1)\omega_{n-1}}\int_{M}(4|\nabla \psi |^2+\scal|\psi|^2)\dV.
\end{align*}
Here, $\psi$ is a harmonic spinor asymptotic to a constant spinor with norm $1$ at $\infty$.
For the definitions of the mass $m$ and asymptotic flatness, we refer to \cite{lee2021geometric}.
One drawback of this formula is that it needs non-negative scalar curvature since otherwise the existence of $\psi$ is not guaranteed. 
Having such a mass formula without the $\scal \ge0$ requirement has been for instance exploited in \cite[Theorem 1.5]{hirsch2020mass}.
Interestingly, \cref{main_integral_inequality} implies such a Witten-type integral formula without imposing any conditions on the scalar curvature.

Given $k\ge1$ and $s>n-2$, we say that an asymptotically flat manifold $(M^n,g)$ is $\Ct^{k}_{-s}$-asymptotically Schwarzschild of mass $m$ if in the asymptotically flat coordinate system
\begin{align}\label{eq AS}
  \left(g-\left(1-\frac{2m}{r^{n-2}}\right)^{-1}\D r^2-r^2g_{\Sphere^{n-1}}\right)\in \Ct^{k}_{-s},
\end{align}
where we refer to \cite{lee2021geometric} for the definition of the weighted function spaces $\Ct^{k}_{-s}$.
Moreover, we denote by $M_r$ the portion of $M$ where, in the asymptotically flat coordinate system, $|x|\le r$. 

\begin{theorem}
    Let $\delta>0$ and let $(M^n,g)$ be a smooth spin manifold which is $\Ct^{1}_{-(n-2)-\delta}$-asymptotically Schwarzschild of mass $m$.
    Then for every radius $r\gg1$ there exists a harmonic spinor $\psi_r$ on $M_r$ normalized so that $\|\psi_r\|^2_{\Lp^2(\partial M_r)} = |\partial M_r|$ such that
    \begin{align*}
         m\ge\frac{1}{2(n-1)\omega_{n-1}}\int_{M_r}(4|\nabla \psi_r|^2+\scal|\psi_r|^2) \dV+\bigO(r^{-\delta}).
    \end{align*}
\end{theorem}

Since manifolds which are asymptotically Schwarzschild are dense within arbitrary asymptotically flat manifolds, the above theorem implies the Riemannian positive mass theorem \cite[Theorem 7]{Bray1997Thesis}.
We also remark that the spinors in our proof live in a different bundle compared to the ones in Witten's argument.

\begin{proof}
Let $r\gg1$ and consider the coordinate sphere $\partial M_r$ in the asymptotically flat end. 
The mean curvature of a coordinate sphere of Schwarzschild in static coordinates $(1-\tfrac{2m}{r^{n-2}})^{-1}\D r^2+r^2g_{\Sphere^{n-1}}$ equals $\tfrac{n-1}r\sqrt{1-\tfrac{2m}{r^{n-2}}}=\tfrac{n-1}r(1-\frac{m}{r^{n-2}}+\mathcal O(r^{-(n-1)}))$.
Since $(M^n,g)$ is asymptotically Schwarzschild, we obtain using \eqref{eq AS} 
\begin{align}\label{eq:mean curvature}
    \mean_r=\frac{n-1}r\left(1-\frac{m}{r^{n-2}}+\mathcal O(r^{-(n-2)-\delta})      \right)
\end{align}
for the mean curvature $\mean_r$ of $\partial M_r$.
Similarly, the induced metric $g_r$ on $\partial M_r$ satisfies
\begin{align}\label{coordinate sphere metric}
 g_r=g_{\Sphere^{n-1}_r}+\mathcal O(r^{-(n-2)-\delta}),
\end{align}
where $g_{\Sphere^{n-1}_r}=r^2g_{\Sphere^{n-1}}$ is the round metric on the sphere of radius $r$.
The latter condition implies that the map $f_r=\id:\partial M_r\to S_r$ is Lipschitz with Lipschitz constant $L_r$ satisfying 
\begin{align}\label{Lipschitz constant2}
    L_r=1+\mathcal \bigO(r^{-(n-2)-\delta}).
\end{align}
Now, as described in Section \ref{sec:almost rigidity}, we can extend $f_r$ to a map $F_r$ from $M_r$ to the ball $\Ball_r$ of radius $r$, and solve on the corresponding twisted spin bundle \(S\) the Dirac equation $\Dirac \psi_r=0$ with boundary condition \(\clm(\nu_{M_r}) \psi_r =  \clms(\nu_{\Ball_r}) \psi_r \).
 Combining equations \eqref{eq:mean curvature}, \eqref{coordinate sphere metric} and \eqref{Lipschitz constant2} with \cref{main_integral_inequality}, which in this situation states
    \begin{align*}
        \int_{M_r}(4|\nabla \psi_r|^2+\scal|\psi_r|^2)\dV\le2\int_{\partial M_r}\left(L_r\frac{n-1}r-\mean_r\right)|\psi_r|^2\dS,
    \end{align*}
    the result follows.
\end{proof}

We expect that this approach will also work in the initial data set setting, and we refer to \cite{brendle2023rigidity} for the corresponding Schr\"odinger--Lichnerowicz formula.
Also, compare this to \cite{bray2022harmonic} and \cite{hirsch2022spacetime}, where integral formulas for the mass are obtained without assuming non-negative scalar curvature or the dominant energy condition.

\appendix

\section{The index theorem}\label{appendix_index_theorem}
In this appendix, we exhibit a quick cut-and-paste argument to derive the index formula for the boundary value problem studied in \cref{subsec:spin_maps} for the case of maps to Euclidean domains in any dimension parity.
Suppose from now on that we are in the setting of \cref{subsec:spin_maps}.
Given a vector field \(\eta \in \Vfds(N)\), we consider the following perturbed Dirac operator
\[
    \Dirac_\eta = \Dirac + \clms(\eta).
\]
Expanding the term \(|\Dirac_\eta \psi|^2\) and integrating by parts the term $\langle \clms(\eta)\psi, \Dirac \psi\rangle$, we obtain the following formula
\begin{align}
    \|\Dirac_\eta \psi\|_{\Lp^2(M)}^2 = \|\Dirac \psi \|_{\Lp^2(M)}^2 &+ \int_M \left( \langle (\clms(\eta) \Dirac + \Dirac \clms(\eta))\psi, \psi \rangle + \lvert \eta \rvert^2 \lvert \psi \rvert^2 \right) \dV \label{eq:callias_integral_formula} \\
    &+ \int_{\partial M} \langle \clm(\nu_M) \clms(\eta) \psi, \psi \rangle \dS \nonumber.
\end{align}
Note that \(\clms(\eta) \Dirac + \Dirac \clms(\eta) = \sum_{i=1}^n \clm(e_i) \clms(\nabla_{\D{f}(e_i)} \eta)\) by a similar argument as in the proof of \cref{boundary_dirac_commutation}.
Moreover, if \(\psi\) satisfies the boundary condition \(\clm(\nu_M) \psi = s \clms(\nu_N) \psi\), then 
\begin{align}
\begin{split}
    \langle \clm(\nu_M) \clms(\eta) \psi, \psi \rangle=&
    - \langle \clms(\eta)  \clm(\nu_M) \psi, \psi \rangle\\
    =& - s \langle \clms(\eta) \clms(\nu_N) \psi, \psi \rangle\\
    =& s \langle \clms(\nu_N) \clms(\eta) \psi, \psi \rangle - 2 s \langle \eta, \nu_N \rangle \lvert \psi \rvert^2 \\
    =&  - \langle \clm(\nu_M) \clms(\eta) \psi, \psi \rangle + 2 s \langle \eta, -\nu_N \rangle \lvert \psi \rvert^2,
    \end{split}
    \end{align}
where we have used the Clifford relation \(\clms(v)\clms(w) + \clms(w)\clms(v) = 2 \langle v, w \rangle\).
Thus we conclude that \(\langle \clm(\nu_M) \clms(\eta) \psi, \psi \rangle = s \langle \eta, -\nu_N \rangle |\psi|^2\).

We obtain the following lemma:

\begin{lemma}\label{vectorfield_vanishing_lemma}
    Let \(s \colon \partial N \to \{\pm 1\}\) and suppose there exists a nowhere-vanishing vector field \(\eta \in \Vfds(N)\) such that \(s \langle \eta, -\nu_N \rangle \geq 0\) on \(\partial N\).
    Then the operator \(\Dirac\) subject to the boundary condition \(\clm(\nu_M) \psi = s \clms(\nu_N) \psi \) has vanishing index.
\end{lemma}
\begin{proof}
    We have \(\ind(\Dirac, {s}) = \ind(\Dirac_{\lambda \eta}, s)\) for any \(\lambda \geq 0\) by homotopy invariance of the index.
    But \labelcref{eq:callias_integral_formula} implies that under the boundary condition given by \(s\) the operator \(\Dirac_{\lambda \eta}\) is invertible for \(\lambda\) sufficiently large and hence has vanishing index.
\end{proof}

\begin{lemma}\label{decomposition_theorem}
    Suppose that we have a decomposition \(N = N_0 \cup_{\Upsilon} N_1\) and \(M = M_0 \cup_{\Sigma} M_1\) respecting all structures, where \(M_i\) and \(N_i\) are codimension zero submanifolds with boundary and the gluing happens along the additional boundary components not present in \(M\) and \(N\), denoted by \(\Upsilon \subseteq \partial N_i\) and \(\Sigma \subseteq \partial M_i\).
    We also assume that the map \(f \colon M \to N\) restricts to \(f_i \colon (M_i, \partial M_i) \to (N_i, \partial N_i)\).
    Let \(s \colon \partial N \to \{\pm 1\}\) and extend this to \(s_i \colon \partial N_i \to \{\pm 1\}\) such that \(s_0|_{\Upsilon} = - s_{1}|_{\Upsilon}\).
    Then
    \[
        \ind(\Dirac, s) = \ind(\Dirac_{0},s_0) + \ind(\Dirac_{1},s_1). 
    \]
\end{lemma}
\begin{proof}
    This is a consequence of the splitting theorem of \textcite[Theorem~6.5]{Baer-Ballmann:Guide-Boundary-value-problems-Dirac}, compare also \cite[Theorem~3.6]{scalar_mean_dirac}.
\end{proof}

\begin{theorem}\label{the_index_theorem}
    Let \(N = \Omega\) be a convex domain in \(\R^n\) and \(s = 1\). Then
    \[
        \ind(\Dirac, 1) = \deg(f)  
    \]
\end{theorem}
\begin{proof}
    We proceed in several steps of successive generality.
    
\begin{enumerate*}
    \item \label{case_identity} The case where \(M = N\) is a convex domain in \(\R^n\) with identical metrics and \(f = \id\) reduces to \cref{example:identity} because \(\chi(N) = 1\).
\end{enumerate*}

\begin{enumerate*}[resume]
    \item \label{case_diffeomorphism} Now assume that \(f \colon M \to N\) is a diffeomorphism, where \(N\) is a convex domain in \(\R^n\).
    Then by homotopy invariance of the index we may assume that \(f\) is an isometry.
    Thus we obtain a bundle isomorphism \(\T{M} \oplus f^\ast \T{N} \cong \T{M} \oplus \T{M}\) which preserves the pseudo-Riemannian metric and connection.
    If \(f\) is also orientation preserving, this is an orientation preserving bundle isomorphism and hence lifts to the associated spinor bundles, in which case we are immediately reduced to \labelcref{case_identity} because then \(\ind(\Dirac, 1) = 1 = \deg{f}\).
    Otherwise, if \(f\) is orientation reversing, then the induced isomorphism \(\T{M} \oplus f^\ast \T{N} \cong \T{M} \oplus \T{M}\) is also orientation reversing and the effect on the spinor bundle is a reversal of the \(\Z/2\)-grading.
    Hence \(\ind(\Dirac, 1) = -1 = \deg(f)\).
\end{enumerate*}

\begin{enumerate*}[resume]
    \item \label{case_local_diffeomorphism} Next we consider a map \(f \colon M \to N\) which on each connected component of \(M\) restricts to a diffeomorphism onto the convex domain \(N\).
    Then the desired index formula follows directly from \labelcref{case_diffeomorphism} by additivity over connected components.
\end{enumerate*}

\begin{enumerate*}[resume]
    \item Now we are ready to treat the general case.
    To this end, let \(y_0 \in N\) be an interior point of \(N\) which is a regular value of \(f\). 
    Choose a smooth vector field \(\eta\) on \(N\) which satisfies \(\eta|_{\partial N} = -\nu_N\) and which is equal to the vector field \(y \mapsto y - y_0\) in a disk neighborhood \(N_0 \subset \interior{N}\) around \(y_0\) and such that \(\eta\) has no zeroes other than \(y_0\).
    By possibly shrinking the disk \(N_0\) further, we can also assume that \(f \colon f^{-1}(N_0) \to N_0\) is a diffeomorphism on each connected component of \(M_0 \coloneqq f^{-1}(N_0)\) and \(M_0\) has a smooth boundary \(\Sigma \coloneqq \partial M_0\).
    Let \(N_1\) denote the closure of the complement of \(N_0\) and \(M_1 \coloneqq f^{-1}(N_1)\) as well as \(\Upsilon \coloneqq \partial N_0\).
    Let \(s_i \colon \partial N_i \to \{\pm 1\}\) be defined by setting \(s_0(\Upsilon) = 1\), \(s_1(\Upsilon) = -1\) and \(s_1(\partial N) = 1\).
    Then by construction \(\eta\) has no zeroes on \(N_1\) and we have \(s_1 \langle \eta, -\nu_{N_1} \rangle > 0\) on \(\partial N_1 = \Upsilon \sqcup \partial N\), where \(\nu_{N_1}\) denotes the unit normal of \(\partial N_1\) pointing inside \(N_1\).
    Thus \cref{vectorfield_vanishing_lemma} implies that \(\ind(\Dirac_{1, s_1}) = 0\).
    Moreover, \cref{decomposition_theorem} implies that \(\ind(\Dirac_{s}) = \ind(\Dirac_{0,s_0}) + \ind(\Dirac_{1,s_1}) = \ind(\Dirac_{0,s_0})\).
    So it suffices to show that \(\ind(\Dirac_{0,s_0}) = \deg(f)\), but we have treated this in \labelcref{case_local_diffeomorphism}.
\end{enumerate*}
\end{proof}

\begin{remark}
    By elaborating on this argument further and using the Poincaré--Hopf index formula~\cite{Hopf:Vektorfelder} for vector fields, it is possible to show the general formula \(\ind(\Dirac, 1) = \deg(f) \cdot \chi(N)\) in this way, where \(N\) is not necessarily a convex domain in \(\R^n\).
    See \textcite{tony2024scalar} where this idea is carried out in a different setting.
\end{remark}

\section{A quantitative matrix Hölder inequality}
\label{appendix:matrix_norms}
Let \(W, V\) be real Euclidean vector spaces and assume for simplicity that \(\dim W = \dim V = n < \infty\).
Let \(T \colon W \to V\) be a linear map.
We can always find a polar decomposition \(T = P \circ U\), where \(U \colon W \to V\) is an isometry and \(P \colon V \to V\) is a non-negative self-adjoint operator.
The operator \(P\) is always uniquely determined by \(T\) and given by the formula \(P = \sqrt{T \circ T^\ast}\), whereas the isometry \(U\) is only unique if \(T\) is invertible.
The eigenvalues \(\sigma_1, \dotsc, \sigma_n \geq 0\) of \(P\) are called the \emph{singular values} of the operator \(T\).
The \emph{trace norm} of \(T\) is defined as \(|T |_{\tr} = \tr(P)\), or in other words as the sum of the singular values. 
The \emph{operator norm} of \(T\), denoted by \(|T|_\op\), is the maximum of its singular values.
More generally, we may also define the \emph{Schatten norms} by \(|T|_p = \tr(P^p)^{1/p}\) for any $1 \leq p < \infty$.
For \(p = 2\), this is known as the \emph{Hilbert--Schmidt norm}.
Given linear maps \(T \colon W \to V\) and \(S \colon V \to V\), we have the following \enquote{Hölder inequality}
\begin{equation}\label{eq:trace_hoelder}
    |S \circ T|_{1} \leq |S|_{p}\ |T|_{q},
\end{equation}
whenever \(1/p + 1/q = 1\).

We will only use the case \(p=1, q = \infty\) and when \(S\) is a non-negative endomorphism. 
In this case we have the following quantitative estimate which we will need in our (almost) rigidity argument:

\begin{lemma}\label{trace_inequality_polar_decomposition}
    Let \(V,W\) be finite dimensional real Euclidean vector spaces of the same dimension, \(T \colon W \to V\) a linear map and \(S \colon V \to V\) a self-adjoint non-negative endomorphism.
    Let \(\sigma_{\min} = \sigma_{\min}(S) \geq 0\) denote the smallest eigenvalue of \(S\).
    Furthermore, let \(P \circ U\) be a polar decomposition of \(S \circ T\), that is, \(P \circ U = S \circ T\) with \(U \colon W \to V\) an isometry and \(P \colon V \to V\) a self-adjoint non-negative endomorphism.
    Then for any \(\lambda \geq |T|_{\op}\), we have
    \[
    \tfrac{\sigma_{\min}}{2} \lvert T - \lambda U \rvert_2^2 \leq \lvert S \rvert_{\tr}\ \lambda^2 - \lvert S \circ T \rvert_{\tr} \lambda.
    \]
    In particular, if \(T \neq 0\),
    \[
    \tfrac{\sigma_{\min}}{2 |T|_{\op}} \lvert T - |T|_{\op} U \rvert_2^2 \leq \lvert S \rvert_{\tr}\ \lvert T \rvert_{\op} - \lvert S \circ T \rvert_{\tr}.
    \]
\end{lemma}
\begin{proof} Let $(e_i)$ be an orthonormal basis of $W$. Then
    \begin{align*}
        \tfrac{\sigma_{\min}}{2} \lvert T - \lambda U \rvert_2^2 &= \tfrac{\sigma_{\min}}{2} \sum_{i} \lvert T(e_i) - \lambda U(e_i) \rvert^2 \\
        &= \tfrac{\sigma_{\min}}{2} \sum_{i} \lvert T(e_i) \rvert^2 + \lambda^2 \lvert U(e_i) \rvert^2 - 2 \lambda \langle T(e_i), U(e_i) \rangle \\
        &\leq \sigma_{\min} \sum_{i} \left( \lambda^2 - \lambda \langle U^\ast T(e_i), e_i \rangle\right) \\
        &= \sigma_{\min} \tr \left(\lambda^2 \id_W - \lambda U^\ast \circ T \right) \\
        &= \sigma_{\min} \tr \left(\lambda^2 \id_V - \lambda T \circ U^\ast \right) \\
        &\leq \tr \left( \lambda^2 S  - \lambda S \circ T \circ U^\ast \right) \\
        &= \tr \left( \lambda^2 S  - \lambda P \right) = \lvert S \rvert_{\tr}\ \lambda^2 - \lvert S \circ T \rvert_{\tr} \lambda . \qedhere
    \end{align*}
\end{proof}

\printbibliography
\end{document}